\documentclass[12pt,a4paper]{article}
\usepackage[utf8]{inputenc}
\usepackage[english]{babel}
\usepackage{amsmath}
\usepackage{mathtools}
\usepackage{amsfonts}
\usepackage{amssymb}
\usepackage[left=3cm,right=3cm,top=3cm,bottom=3cm]{geometry}
\usepackage{enumerate}
\usepackage{nicefrac}
\usepackage{graphicx}
\usepackage{fancyhdr}
\usepackage{esvect}
\usepackage{cancel}
\usepackage{amsthm}
\usepackage{hyperref}
\usepackage{array}
\usepackage{pifont}
\usepackage{wallpaper}
\usepackage{color}
\usepackage{rotating}
\usepackage{tikz-cd}
\usepackage{enumitem}
\usepackage{courier}
\usepackage[T1]{fontenc}
\usepackage[nospace,noadjust]{cite}
\usepackage[nottoc]{tocbibind}

\newtheorem{thm}{Theorem}[section]

\newtheorem{lem}[thm]{Lemma}

\newtheorem{kor}[thm]{Corollary}
\newtheorem{prop}[thm]{Proposition}

\newtheorem*{thma}{Theorem A}
\newtheorem*{thmb}{Theorem B}
\newtheorem*{thmc}{Theorem C}

\theoremstyle{definition}
\newtheorem{defn}[thm]{Definition}

\newtheorem{bem}[thm]{Remark}
\newtheorem*{varbem}{Remark}

\newtheorem{bems}[thm]{Remarks}
\newtheorem*{varbems}{Remarks}

\newtheorem*{notat}{Notation}

\newcommand{\mf}{\mathfrak}
\newcommand{\mc}{\mathcal}

\newcommand{\lra}{\longrightarrow}

\newcommand{\lmt}{\longmapsto}
\newcommand{\bb}{\mathbb}
\newcommand{\ssq}{\subseteq}

\newcommand{\ssnq}{\subsetneq}

\newcommand{\abs}[1]{\left\lvert#1\right\rvert}

\newcommand{\lrk}{\left(}
\newcommand{\rrk}{\right)}

\DeclareMathOperator{\im}{im}

\DeclareMathOperator{\Hom}{Hom}
\DeclareMathOperator{\Gal}{Gal}

\DeclareMathOperator{\rk}{rk}
\DeclareMathOperator{\Pic}{Pic}

\DeclareMathOperator{\Ann}{Ann}
\DeclareMathOperator{\ord}{ord}

\renewcommand{\Re}{\mathrm{Re}}

\numberwithin{equation}{section}

\newcolumntype{C}[1]{>{\centering}m{#1}}

\author{Pascal Stucky}

\title{Annihilators of the ideal class group of a cyclic extension of a global function field}
\begin{document}
\maketitle

\begin{abstract}Let $K$ be a global function field and fix a place $\infty$ of $K$. Let $L/K$ be a finite real abelian extension, i.e. a finite, abelian extension such that $\infty$ splits completely in $L$. Then we define a group of elliptic units $C_L$ in $\mc O_L^\times$ analogously to Sinnott's cyclotomic units and compute the index $[\mc O_L^\times:C_L]$. In the second part of this article, we additionally assume that $L$ is a cyclic extension of prime power degree. Then we can use the methods from Greither and Ku\v{c}era to take certain roots of these elliptic units and prove a result on the annihilation of the $p$-part of the class group of $L$.\end{abstract}
\ \\

\textbf{MS Classification:} 11R20, 11R58, (11R27, 11G09).

\section{Introduction}

In \cite{thaine1988ideal} F.~Thaine studied the relation of the ideal class group $Cl(L)$ of a totally real absolutely abelian number field $L$ and a certain group of cyclotomic units introduced by W.~Sinnott in \cite{sinnott1980stickelberger}. These can be used to produce annihilators of the $p$-Sylow subgroup $Cl(L)_p$ of the ideal class group. His methods were generalized by K.~Rubin to abelian extensions of an imaginary quadratic base field $K$ in \cite{rubin1987global}, where the cyclotomic units are replaced by elliptic units. These approaches are closely related to Kolyvagin's Euler system machinery, in fact, Rubin already works with so called \emph{special numbers} in a quite general setting and then specialises to the case of an imaginary quadratic base field. This method yields nice results when $p$ does not divide $[L:\bb Q]$ (resp. $[L:K]$), but when $p$ is a divisor of the degree of the extension, the obtained annihilation statement is not satisfying. \\

When $L/\bb Q$ is a cyclic extension of degree $p^k$, the ideal class group (considered as the Galois group of the Hilbert class field of $L$) splits into a genus part (corresponding to the extension $F_I/L$, where $F_I$ denotes the genus field of $L$) and a so-called \emph{non-genus part}. In order to study this non-genus part $(\sigma-1)Cl(L)_p$, where $\sigma$ denotes a generator of $\Gal(L/\bb Q)$, C.~Greither and R.~Ku\v{c}era extended Rubin's method in \cite{greither2004annihilators} and \cite{greither2006annihilators} to so-called \emph{semispecial numbers}. These satisfy weaker conditions but are still sufficient to produce annihilators. The source of semispecial numbers in this case are certain roots with respect to group ring valued exponents of Sinnott's cyclotomic units.\\

It was shown by D.~Burns and A.~Hayward that the annihilation result of Greither and Ku\v{c}era can also be deduced from the equivariant Tamagawa number conjecture (see \cite{burns2007explicit}), however, the proof of Greither and Ku\v{c}era is constructive whereas the method of Burns and Hayward uses abstract arguments. In particular, the explicit construction of the roots of circular units enables Greither and Ku\v{c}era to refine their method in \cite{greither2015annihilators} and weaken the conditions on $L$ to cover even more cases. They use results on Sinnott's module from \cite{greither2014linear} which are formulated in an abstract way without specifying a certain extension of number fields. Hence, these results can also be used in other cases. This is done by H.~Chapdelaine and R.~Ku\v{c}era in \cite{chapdelaine2017annihilators}, where they prove an annihilation result for a cyclic extension of an imaginary quadratic field of prime power degree. They take roots of elliptic units studied by H.~Oukhaba in \cite{oukhaba2003index} to obtain semispecial numbers and can then adapt the methods of Greither and Ku\v{c}era to this case. \\

In this article, we want to apply the methods described above to the case of global function fields. For this purpose, we explicitly construct elliptic units based on the torsion points of rank-1 sign-normalized Drinfeld modules as in \cite{hayes1985stickelberger}. As in the case of cyclotomic units (see e.g. \cite{kucera2005circular}) there are several methods to construct elliptic units in an arbitrary real abelian extension of global function fields. We use the function field version of Sinnott's cyclotomic units and are hence able to prove an index formula for this subgroup of the units of $L$ analogously to the ones in the rational case (\cite{sinnott1980stickelberger}) and in the imaginary quadratic case (\cite{oukhaba2003index}). There exist some other index formulas for elliptic units in function fields, e.g.~by L.~Yin (\cite{yin_1997},\cite{yin1997index}) who studied cyclotomic units in ray class fields in the sense of L.~Washington (\cite{washington1997introduction}) or by H.~Oukhaba (\cite{oukhaba1992groups}, \cite{oukhaba1995construction}, \cite{oukhaba1997groups}) who studied elliptic units in extensions where at most one prime ideal ramifies in $L/K$. However, there is no discussion of an index formula for a general abelian extension of global function fields known to the author. Moreover, we can use the methods of Greither and Ku\v{c}era to extract roots from the defined elliptic units and obtain an annihilation result similar to the one in the number field case.\\

The article has the following structure: For the convenience of the reader, we first present a collection of the necessary notation and state the main results afterwards. Then we introduce the elliptic units and prove an index formula for them (Sections \ref{sec:oukhabastart} - \ref{sec:oukhabaend}). This part will closely follow \cite{oukhaba2003index}. The rest of the article (Sections \ref{sec:CK} - \ref{sec:CKend}) will deal with the desired annihilation result and will have the same structure as \cite{chapdelaine2017annihilators}.
\pagebreak

\subsection{Notation and Preliminaries}\label{sec:notat1}

Let $K$ be a global function field with constant field $\bb F_q$ and let $\infty$ be a fixed place of $K$ of degree $d_\infty$. We define
\begin{itemize}
\item $\mc O_K$ is the ring of functions in $K$ which have no poles away from $\infty$,

\item $h(K)$ (resp.~$h:=h_K$) is the class number of $K$ (resp.~$\mc O_K$), i.e.~the number of elements in $\Pic(K)$ (resp.~$\Pic(\mc O_K)$). Note that $h=h(K)d_\infty$.

\item $w_\infty:=q^{d_\infty}-1$,

\item $\ord_\infty$ is the valuation at $\infty$,

\item $K_\infty$ is the completion of $K$ at $\infty$,

\item $\bb F_\infty$ is the constant field of $K_\infty$,

\item for any prime $\mf p$ of $K$ set $N\mf p:=q^{\deg(\mf p)}$. This is the order of the residue class field at $\mf p$.
\end{itemize}

An extension of $K$ is called a \emph{real} extension, if it is contained in $K_\infty$. Now let $\rho$ be a sign-normalized rank-1 Drinfeld module w.r.t. a fixed sign-function $sgn$. Then we set $K_{(1)}$ to be the extension of $K$ generated by all coefficients of $\rho_x$, $x\in \mc O_K$. Note that this extension is finite. For any integral ideal $\mf m\ssq \mc O_K$
\begin{itemize}
\item $\rho_\mf m$ is the generator of the principal ideal generated by the elements $\rho_x$ for all $x\in \mf m$,

\item $\Lambda_{\mf m}$ is the set of $\mf m$-torsion points of $\rho$,

\item $K_\mf m:=K_{(1)}(\Lambda_\mf m)$,

\item $H_\mf m$ is the maximal real subfield of $K_\mf m$ and is called the \emph{real ray class field} of $K$ modulo $\mf m$ (in particular $H=H_{(1)}$ is the \emph{real Hilbert class field} of $K$),

\item $H_{\mf m^\infty}:=\bigcup_{n\geq 1}H_{\mf m^n}$.
\end{itemize}

For any finite extension $L/K$
\begin{itemize}
\item $\mc O_L$ is the integral closure of $\mc O_K$ in $L$,

\item $\mu(L)$ is the group of roots of unity in $L$,

\item $w_L:=\abs{\mu(L)}$,

\item $h_L$ is the class number of $\mc O_L$,

\item if $\mf p\ssq \mc O_K$ is a prime ideal, then $\mf p_L$ is the product of all ideals of $\mc O_L$ above $\mf p$,

\item if $L/K$ is abelian and $\mf m$ is an integral ideal of $K$, then we set $L_\mf m=L\cap H_\mf m$.
\end{itemize}
Note that $w_K=q-1$.

\begin{bem}
It is shown in \cite[§3, §4]{hayes1985stickelberger} that
\begin{enumerate}[label=(\roman*)]
\item $w_{H_\mf m}=w_\infty$ for all $\mf m$ \cite[§3]{hayes1985stickelberger}, so $\bb F_\infty$ is the constant field of $H_\mf m$,

\item $[H_\mf m : K]=\frac{h}{w_K} \abs{(\mc O_K/\mf m)^\times}$ \cite[Eq. (3.2)]{hayes1985stickelberger} for $\mf m\neq (1)$ and $[H:K]=h$,

\item $[K_\mf m :H_\mf m]=w_\infty$ for $\mf m\neq 1$\cite[§4]{hayes1985stickelberger} and $[K_{(1)}:H]=\frac{w_\infty}{w_K}$ \cite[Cor. 4.8(2)]{hayes1985stickelberger}.

\end{enumerate}
\end{bem}

Now suppose that the extension $L/K$ is Galois and $\mf p$ is a prime of $K$. Then
\begin{itemize}
\item $D_\mf P\ssq \Gal(L/K)$ is the decomposition group of a prime $\mf P$ of $L$ above $\mf p$. If $L/K$ is abelian, this subgroup does not depend on the choice of the prime $\mf P$, hence we write $D_\mf p$ in this case.

\item $T_\mf P\ssq D_\mf P$ is the inertia subgroup. If $L/K$ is abelian we write again $T_\mf p$.

\item $(\mf P,L/K)$ (or $\sigma_\mf P$ if the extension is clear) is a lift to $\Gal(L/K)$ of the corresponding Frobenius element in $D_\mf P/T_\mf P$. These elements form a conjugacy class in $\Gal(L/K)$ which will be denoted by $(\mf p,L/K)$ (or $\sigma_\mf p$). If $L/K$ is abelian and $\mf p$ is unramified, this conjugacy class contains only one element which coincides with the Artin symbol.
\end{itemize}

For any abelian group $G$, we set
\begin{align*}
\widehat{G}:=\Hom(G, \bb C^\times)
\end{align*}
to be the group of characters of $G$. For any subset $U\ssq G$, we define
\begin{align*}
s(U):=\sum_{\sigma\in U} \sigma\in \bb Z[G].
\end{align*}
If $U$ is a subgroup of $G$, we define the associated idempotent
\begin{align*}
e_U:=\frac{1}{\abs{U}}s(U)\in \bb Q[G].
\end{align*}
To a character $\chi\in \widehat{G}$, we also assign an idempotent
\begin{align*}
e_\chi:=\frac{1}{\abs{G}}\sum_{\sigma\in G} \chi(\sigma)\sigma^{-1}\in \bb C[G].
\end{align*}
By extension of scalars with $\bb Z\ssq R\subseteq \bb C$, we can evaluate a character $\chi\in \widehat{G}$ at an element $a=\sum_{\sigma\in G} a_\sigma \sigma\in R[G]$, i.e. we set
\begin{align*}
\chi(a)=\sum_{\sigma\in G} a_\sigma\chi(\sigma)\in \bb C.
\end{align*}
Finally, for any multiplicative abelian group $A$ and a positive integer $m$, we set $A/m:=A/(A^m)$.

\subsection{Main results}

Let $L/K$ be a finite real abelian extension with Galois group $G$. Then the elliptic units $C_L$ of $L$ are essentially the norms of torsion points in $K_\mf m$ together with certain unramified units (for a precise definition see Section \ref{sec:groupell}). These form a subgroup of $\mc O_L^\times$ which has finite index given by

\begin{thma}
We get
\begin{align*}
[\mc O_L^\times:C_L]=\frac{(h_Kw_\infty)^{[L:K]-1}w_Kh_L}{w_Lh_K}\frac{\prod_\mf p [L\cap H_{\mf p^\infty}:L_{(1)}]}{[L:L_{(1)}]}\frac{[\bb Z[G]:U']}{d(L)}\ .
\end{align*}
\end{thma}

\begin{varbems}
\begin{itemize}
\item The Sinnott module $U'\ssq \bb Q[G]$ is defined in \cite{sinnott1980stickelberger}. The Sinnott index $[\bb Z[G]:U']$ as well as the number $d(L)$ can be computed in certain cases (cf.~Remark \ref{rem:dL} and Proposition \ref{prop:RU}).

\item This index formula is the analogue of \cite[Thm. 4.1]{sinnott1980stickelberger} and \cite[Thm. 1]{oukhaba2003index}. It is proven as Theorem \ref{thm:indexformula}.
\end{itemize}
\end{varbems}

Now let $p$ be an odd prime not dividing the class number $h_K$ of $\mc O_K$, the characteristic of $K$ and the number of roots of unity $w_K$ of $K$. Suppose that $L/K$ is cyclic of degree $p^k$ for some $k>0$, then we can define $\mc C_L\ssq \mc O_L^\times$ satisfying $\mc C_L^{h_K}\cdot \mu(L)=C_L$. Let $\eta$ be a top generator of $\mc C_L$ (precisely defined in Section \ref{sec:CK}) and $\mf p_1,...,\mf p_s$ be the primes of $K$ which ramify in $L$. We assume $s\geq 2$. Let $\sigma$ be a generator of $\Gal(L/K)$. Then there is a certain subextension $K\ssq L'\subset L$ such that we get

\begin{thmb}
Define $y:=\prod_{i=2}^s (1-\sigma^{n_i})$, where $n_i$ is the index of the decomposition group of $\mf p_i$ in $\Gal(L/K)$. Then there exists a unique $\alpha\in L$ with $\eta=\alpha^y$ and $N_{L/L'}(\alpha)=1$.
\end{thmb}

\begin{varbems}
\begin{itemize}
\item The element $\alpha$ from Theorem B is a semispecial number in the sense of Definition \ref{defn:semispecial} (cf. Theorem \ref{thm:19}).

\item This result is the analogue of \cite[Thm. 1.2]{greither2015annihilators} and \cite[4.2]{chapdelaine2017annihilators}. The field $L'$ is defined right before Theorem \ref{thm:rootofunit}.
\end{itemize}
\end{varbems}

Now we can extend $\mc C_L$ by $\alpha_j$ (taking a root for each subextension $K\ssq L_j\ssq L$) to obtain $\overline{\mc C_L}$. In this special case, the index formula from Theorem A simplifies significantly and we obtain
\begin{align*}
[\mc O_L^\times:\overline{\mc C_L}]=w_\infty^{p^k-1}\cdot \frac{h_L}{h_K}\cdot \varphi_L^{-1}
\end{align*}
for a certain $p$-power $\varphi_L$ (cf. Theorem \ref{thm:enlarged}). In particular, $[\overline{\mc C_L}:C_L]=p^\nu$ where $\nu$ is determined by the $n_i$ (also see Theorem \ref{thm:enlarged}). Our main result then reads

\begin{thmc}
There exists a certain number $0\leq r<k$ such that
\begin{align*}
\Ann_{\bb Z[\Gal(L/K)]}((\mc O_L^\times/\overline{\mc C_L})_p)\ssq \Ann_{\bb Z[\Gal(L/K)]}((1-\sigma^{p^r})Cl(\mc O_L)_p)\ .
\end{align*}
\end{thmc}

\begin{varbem}
This is the analogue of \cite[Thm. 5.3]{greither2015annihilators} and \cite[Thm. 7.5]{chapdelaine2017annihilators}. The number $r$ has a concrete description given in Theorem \ref{thm:main}.
\end{varbem}

\subsection{Acknowledgements}

First I want to thank W.~Bley for introducing me to this topic and for his support while working on this article. I am also very grateful for many discussions with M.~Hofer during this project. Moreover, I want to thank H.~Oukhaba and R.~Ku\v{c}era for answering my questions.

\section{Elliptic units in global function fields}\label{sec:oukhabastart}

Let $\Omega$ be the completion of the algebraic closure of $K_\infty$ and let $\Gamma$ be a lattice in $\Omega$, i.e. a finitely generated projective $\mc O_K$-module. The \emph{exponential function} associated to $\Gamma$ is defined by
\begin{align*}
e_\Gamma:\Omega&\lra \Omega\\
z&\lmt z\prod_{\substack{\gamma\in \Gamma\\ \gamma\neq 0}} \lrk 1-\frac{z}{\gamma}\rrk.
\end{align*}
We say that $\Gamma$ is \emph{special}, if the rank-1 Drinfeld module associated to $\Gamma$ (see \linebreak \cite[§5]{hayes1985stickelberger}) is sign-normalized w.r.t. the fixed sign-function $sgn$. For each $\Gamma$, there exists an invariant $\xi(\Gamma)\in \Omega^\times$ such that $\xi(\Gamma)\Gamma$ is special. This invariant is unique up to multiplication by an element of $\bb F_\infty$. 

\subsection{Unramified elliptic units}

Following \cite[Sec. 2]{oukhaba1997groups}, we can fix a fractional ideal $\mf c$ of $K$ and a choice of the invariant $\xi(\mf c)$ such that the sign-normalized rank-1 Drinfeld module associated to $\Gamma:=\xi(\mf c)\mf c$ is exactly $\rho$. Let $D$ be the differential of the twisted polynomial ring (see e.g. \cite[§4]{hayes1985stickelberger}). Then for any non-zero integral ideal $\mf a$ of $K$, the rank-1 Drinfeld module associated to $D(\rho_\mf a)\mf a^{-1}\Gamma$ is sign-normalized w.r.t $sgn$, hence we can choose
$\xi(\mf a^{-1}\mf c)=D(\rho_\mf a)\xi(\mf c)$. Any fractional ideal of $K$ is of the form $\mf d=\mf a\mf b^{-1}\mf c$ and setting $\tau:=(\mf d^{-1}\mf c, K_{(1)}/K)$ we can define
\begin{align*}
\xi(\mf d)=\frac{D(\rho_\mf b)}{D(\rho_\mf a)^\tau}\xi(\mf c).
\end{align*}

\begin{lem}
The element $\xi(\mf d)$ is well-defined, i.e. it is independent of the choice of $\mf a$ and $\mf b$. It depends on the choice of $\mf c$ and $\xi(\mf c)$.
\end{lem}

\begin{proof}
Suppose that $\mf d=\mf a\mf b^{-1}\mf c=\mf a'\mf b'^{-1}\mf c$. This implies $\mf a\mf b'=\mf a'\mf b$ and hence 
\begin{align*}
\rho_{\mf a\mf b'}=\rho_{\mf a'\mf b}.
\end{align*}
The ideal class group acts on the set of isomorphism classes of rank-1 Drinfeld modules and via this action we obtain (cf.~\cite[Prop. 13.15]{rosen2013number})
\begin{align*}
\rho_{\mf a\mf b'}\rho_{\mf a'}^{\sigma_{\mf a\mf b'}}&=\mf \rho_{\mf a\mf a'\mf b'}=\rho_{\mf b'}^{\sigma_{\mf a\mf a'}}\mf \rho_{\mf a\mf a'},\\
\rho_{\mf a'\mf b}\rho_{\mf a}^{\sigma_{\mf a'\mf b}}&=\mf \rho_{\mf a\mf a'\mf b}=\rho_{\mf b}^{\sigma_{\mf a\mf a'}}\mf \rho_{\mf a\mf a'}.
\end{align*}
Since $\mf a\mf a'\neq 0$ (we only consider nonzero ideals), we have $D(\rho_{\mf a\mf a'})\neq 0$. Because of $\sigma_{\mf a\mf b'}=\sigma_{\mf a'\mf b}=\tau\sigma_{\mf a\mf a'}$, we get
\begin{align*}
\lrk\frac{D(\rho_{\mf b})}{D(\rho_{\mf a})^\tau}\rrk^{\sigma_{\mf a\mf a'}}&=\frac{D(\rho_\mf b^{\sigma_{\mf a\mf a'}})}{D(\rho_{\mf a}^{\sigma_{\mf a'\mf b}})}=\frac{D(\rho_{\mf a'\mf b})}{D(\rho_{\mf a\mf a'})}\\
&=\frac{D(\rho_{\mf a\mf b'})}{D(\rho_{\mf a\mf a'})}=\frac{D(\rho_{\mf b'}^{\sigma_{\mf a\mf a'}})}{D(\rho_{\mf a'}^{\sigma_{\mf a\mf b'}})}=\lrk\frac{D(\rho_{\mf b'})}{D(\rho_{\mf a'})^\tau}\rrk^{\sigma_{\mf a\mf a'}}.
\end{align*}
\end{proof}

With these definitions, we obtain the following explicit form of the principal ideal theorem

\begin{lem}[{\cite[Lemma 3]{oukhaba1997groups}}]\label{lem:ouk}
Let $\mf d_1,\mf d_2$ and $\mf d$ be fractional ideals of $K$. Then the ideal $\mf d_2\mf d_1^{-1}\mc O_{K_{(1)}}$ is principal generated by $\xi(\mf d_1)/\xi(\mf d_2)$. Moreover, we have
\begin{align*}
\lrk \frac{\xi(\mf d_1)}{\xi(\mf d_2)}\rrk^{(\mf d,K_{(1)}/K)}=\frac{\xi(\mf d_1\mf d^{-1})}{\xi(\mf d_2\mf d^{-1})}.
\end{align*}
\end{lem}

Now let $\sigma\in \Gal(H/K)$ be arbitrary and let $\mf a\ssq \mc O_K$ be such that $(\mf a^{-1},H/K)=\sigma$. Let $x\in \mc O_K$ be a generator of the principal ideal $\mf a^{h}$, then we can define
\begin{align*}
\partial(\sigma):=\lrk x\xi(\mf a)^{h}\rrk^{w_\infty/w_K}.
\end{align*}
\begin{bem}\label{rem:w_K}
\begin{enumerate}[label=(\roman*)]
\item The element $\partial(\sigma)^{w_K}$ is well-defined, i.e. it is independent of the choice of $\mf a$ and $x$. Indeed, it is even independent of the choice of $\mf c$ and $\xi(\mf c)$: If $\mf c'$ and $\xi'(\mf c')$ was used to define invariants $\xi'(\mf d)$ for any fractional ideal $\mf d$, then $\xi'(\mf d)\mf d$ would again correspond to a sign-normalized rank-1 Drinfeld module. Since these lattices only differ by an element of $\mu(H)$ (see e.g. \cite[Sec.~2]{oukhaba1997groups}), we obtain $\xi(\mf d)=\zeta\xi'(\mf d)$ for some $\zeta\in \mu(H)$. Taking the $w_\infty$-th power kills the root of unity, so the element $\partial(\sigma)^{w_K}$ will be the same.

\item The above definition differs from the one given in \cite{oukhaba1997groups} by the factor $1/w_K$ in the exponent. This definition of $\partial(\sigma)$ still depends on the choice of the generator $x$ and of the ideal $\mf c$ and $\xi(\mf c)$. However, two different choices only differ by an element of $\mu(K)$. Since we are only interested in subgroups of the units containing $\mu(K)$, it suffices to define $\partial(\sigma)$ "up to roots of unity".
\end{enumerate}
\end{bem}

\begin{lem}
Let $\sigma,\sigma_1,\sigma_2\in \Gal(H/K)$. Then $\frac{\partial(\sigma_1)}{\partial(\sigma_2)}\in \mc O_{H}^\times$ and
\begin{align*}
\lrk \frac{\partial(\sigma_1)}{\partial(\sigma_2)}\rrk^\sigma=\frac{\partial(\sigma_1\sigma)}{\partial(\sigma_2\sigma)}\ .
\end{align*}
\end{lem}

\begin{proof}
This follows directly from Lemma \ref{lem:ouk} and $[\mc O_{K_{(1)}}^\times :\mc O_H^\times] =\frac{w_\infty}{w_K}$ (see \cite[Lemma 1.5 (1)]{yin1997index}). 
\end{proof}

\subsection{Ramified elliptic units}\label{sec:ramunits}

Using the exponential function, we can define the element
\begin{align*}
\lambda_\mf m:=\xi(\mf m)e_\mf m(1)
\end{align*}
for each integral ideal $\mf m\neq (1)$. It is shown in \cite[§5]{hayes1985stickelberger} that this element is a generator of the $\mf m$-torsion points $\Lambda'_\mf m$ of the sign-normalized rank-1 Drinfeld module $\rho'$ associated to $\xi(\mf m)\mf m$. The construction of $K_\mf m$ does not depend on the chosen Drinfeld module but only on the sign-function, hence $\lambda_\mf m\in K_{(1)}(\Lambda'_\mf m)=K_\mf m$ (cf.~\cite[§4]{hayes1985stickelberger}). Indeed, if $\mf b$ is an integral ideal of $\mc O_K$ such that $\mf b$ is prime to $\mf m$ and $(\mf b,K_{(1)}/K)=(\mf m^{-1},K_{(1)}/K)$ one can show that $(\mf b\mf c,K_\mf m/K)$ defines a bijection $\Lambda_\mf m\lra \Lambda'_\mf m$ (note that $\xi(\mf m)\mf m$ is associated to the Drinfeld module $\mf b\mf c*\rho$, then use \cite[Thm.~4.12]{hayes1985stickelberger}). It is also shown in \cite[Thm.~4.17]{hayes1985stickelberger} that
\begin{align*}
\alpha_\mf m:=-N_{K_\mf m/H_\mf m}(\lambda_\mf m)=\lambda_{\mf m}^{w_\infty}\in H_\mf m
\end{align*}
is a unit if $\mf m$ is not a prime power and that $\alpha_{\mf p^k}$ generates the ideal $\mf p_{H_\mf m}^{w_\infty/w_k}$.

\begin{bem}
\begin{enumerate}[label=(\roman*)]
\item The element $\lambda_\mf m$ depends on the choice of $\mf c$ which was used to define the invariants $\xi(\mf m)$. As already noted in Remark \ref{rem:w_K}, changing $\mf c$ would change $\xi(\mf m)$ by a root of unity in $H$, therefore $\alpha_\mf m=\lambda_\mf m^{w_\infty}$ is independent of this choice.

\item Note that our definition of $\alpha_\mf m$ differs from the one in \cite{hayes1985stickelberger} by a sign. This is neccessary for obtaining the correct norm relation, see Proposition \ref{prop:normrel} below.

\end{enumerate}
\end{bem}

\subsection{The group of elliptic units in an arbitrary real abelian extension}\label{sec:groupell}

Now let $L$ be a finite real abelian extension of $K$ of conductor $\mf m$. Recall that for any integral ideal $\mf n\ssq \mc O_K$ we defined $L_\mf n=L\cap H_\mf n$. Set
\begin{align*}
\varphi_{L,\mf n}:=N_{H_\mf n/L_\mf n}(\alpha_\mf n)^{h}\ .
\end{align*}

\begin{bem}\label{rem:hpower}
Raising to the $h$-th power is neccessary to ensure compatibility with the unramified elliptic units for the desired index formula. If there are no unramified elliptic units (e.g. when $L/K$ is a totally ramified extension), we can also work with the elements $\eta_\mf n=\varphi_{L,\mf n}^{1/h}$, see Section \ref{sec:CK}.
\end{bem}

\begin{kor}
\begin{enumerate}[label=(\roman*)]
\item If $\mf n$ is not a prime power, then $\varphi_{L,\mf n}\in \mc O_{L_\mf n}^\times$.

\item If $\mf n=\mf p^k$, then $\varphi_{L,\mf n}$ generates the ideal $\mf p_{L_\mf n}^{[H:L_{(1)}]hw_\infty/w_K}$.

\end{enumerate}
\end{kor}

\begin{proof}
This follows directly from \cite[Thm.~4.17]{hayes1985stickelberger}.
\end{proof}

\begin{defn}\label{defn:ellipticunits}
\begin{enumerate}[label=(\roman*)]
\item For $\sigma_1,\sigma_2\in \Gal(L_{(1)}/K)$ define
\begin{align*}
\frac{\partial_L(\sigma_1)}{\partial_L(\sigma_2)}:=N_{H/L_{(1)}}\lrk\frac{\partial(\widehat{\sigma}_1)}{\partial(\widehat{\sigma}_2)}\rrk,
\end{align*}
where $\widehat{\sigma_i}$ is any lift of $\sigma_i$ to $\Gal(H/K)$.

\item The subgroup $\Delta_L$ of $\mc O_{L_{(1)}}^\times$ generated by $\mu(L)$ and the elements
\begin{align*}
\frac{\partial_L(\sigma_1)}{\partial_L(\sigma_2)}
\end{align*}
for $\sigma_1,\sigma_2\in \Gal(L_{(1)}/K)$ is the \emph{group of unramified elliptic units} of $L$.

\item The elements $\varphi_{L,\mf n}$ for $\mf n\mid \mf m$, $\mf n\neq (1)$ are called the \emph{ramified elliptic numbers} of $L$.

\item The $\Gal(L/K)$-submodule $P_L$ of $L^\times$ generated by $\Delta_L$ and the ramified elliptic numbers is called the \emph{group of elliptic numbers} of $L$.

\item The \emph{group of elliptic units} $C_L$ of $L$ is defined by $C_L:=P_L\cap \mc O_L^\times$.

\end{enumerate}
\end{defn}

\begin{prop}\label{prop:normrel}
We have
\begin{align*}
N_{L_{\mf n\mf p}/L_\mf n}(\varphi_{L,\mf n\mf p})=\begin{cases}
\varphi_{L,\mf n},&\mf p\mid \mf n,\\
\varphi_{L,\mf n}^{1-\sigma_\mf p^{-1}}, \qquad &\mf p\nmid \mf n, \mf n\neq (1),\\
 x_\mf p^{w_\infty/w_K[H:L_{(1)}]}\lrk\frac{\partial_L(1)}{\partial_L(\sigma_\mf p^{-1})}\rrk, &\mf n=(1),
\end{cases}
\end{align*}
where $\sigma_\mf p=(\mf p,L_\mf n/K)$ and $x_\mf p$ is a generator of $\mf p^h$. The last equation should be read modulo roots of unity (cf.~Remark \ref{rem:w_K}).
\end{prop}

\begin{proof}
The first two cases can be deduced from the definition of the elliptic units and the norm relation in \cite[Prop.~2.3]{oukhaba1995construction}.\\

In the case $\mf n=(1)$, we use \cite[Remark 1]{oukhaba1997groups} where he says that
\begin{align*}
N_{K_\mf p/K_{(1)}}(\mu_\mf p)=\frac{\xi(\mf p^{-1}\mf c)}{\xi(\mf c)}
\end{align*}
for a generator $\mu_\mf p$ of $\Lambda_\mf p$. As already discussed in Section \ref{sec:ramunits}, we can choose $\mu_\mf p=\lambda_{\mf p}^{(\mf b\mf c,K_\mf p/K)^{-1}}$, where $\mf b$ is an integral ideal prime to $\mf p$ such that $(\mf b,K_{(1)}/K)=(\mf p^{-1},K_{(1)}/K)$. Then we obtain with Lemma \ref{lem:ouk}
\begin{align*}
N_{K_\mf p/K_{(1)}}(\lambda_\mf p)=N_{K_\mf p/K_{(1)}}(\mu_\mf p)^{(\mf b\mf c,K_{(1)}/K)}=\lrk\frac{\xi(\mf p^{-1}\mf c)}{\xi(\mf c)}\rrk^{(\mf p^{-1}\mf c,K_{(1)}/K)}=\frac{\xi(\mc O_K)}{\xi(\mf p)}.
\end{align*}
With the definitions of $\varphi_{L,\mf p}$ and $\partial_L(\sigma)$ the desired result follows directly.
\end{proof}

\subsection{$L$-functions and the analytic class number formula for function fields}

Let $L$ be an arbitrary finite abelian extension of $K$ and set $G:=\Gal(L/K)$. Let $\chi$ be a character of $G$ and $\mf p$ be a prime of $K$ with decomposition group $D_\mf p$ and inertia group $T_{\mf p}\ssq G$. Recall that $\sigma_{\mf p}\in G$ is a lift of the Frobenius element in $D_\mf p/T_{\mf p}$. We set
\begin{align*}
\chi(\mf p)=\chi(\sigma_\mf p e_{T_\mf p}).
\end{align*}
Note that we have $\chi(\mf p)\neq 0$ if and only if $T_\mf p\ssq \ker(\chi)$.\\

For a finite set of primes $S$ of $K$ we define the $S$-truncated $L$-function $L_S(\chi, s)$ associated to $\chi$ as the Euler product
\begin{align*}
\prod_{\mf p\notin S} (1-\chi(\mf p)N\mf p^{-s})^{-1},\qquad \Re(s)>1,
\end{align*}
where the product runs over all primes of $K$ which are not contained in $S$. 

If $S=\emptyset$, we simply write 
\begin{align*}
L(\chi,s)=L_{\emptyset}(\chi,s).
\end{align*}

If $\chi=1$ we obtain that
\begin{align*}
L_S(1,s)=\zeta_{K,S}(s)
\end{align*}
is the $S$-truncated Dedekind $\zeta$-function of $K$.

We summarize some results on $L$-functions in the next

\begin{prop}\label{prop:Lfunctions}
\begin{enumerate}[label=(\roman*)]
\item $L_S(\chi,s)$ has a meromorphic continuation to the whole complex plane which will also be denoted by $L_S(\chi,s)$. If $K_\chi=L^{\ker(\chi)}$ is not a constant field extension, this continuation is holomorphic.

\item We have
\begin{align*}
\zeta_K(s)=\frac{P(q^{-s})}{(1-q^{-s})(1-q^{1-s})},
\end{align*}
where $P(x)\in \bb Z[x]$ is a polynomial satisfying $P(0)=1$ and $P(1)=h(K)$.

\item If $L'\supseteq L$ is a finite abelian extension of $K$ with Galois group $G'$ and $\psi$ is the inflation of $\chi$ to $G'$, we have
\begin{align*}
L_S(\chi,s)=L_S(\psi,s),
\end{align*}
i.e. the $L$-function is invariant under inflation.

\item We have
\begin{align*}
\zeta_{L}(s)=\zeta_{K}(s)\cdot \prod_{\chi\neq 1} L(\chi,s),
\end{align*}
where the product runs over all non-trivial characters of $G$.
\end{enumerate}
\end{prop}

\begin{proof}
\begin{enumerate}[label=(\roman*)]
\item This is \cite[Thm.~9.25]{rosen2013number}.

\item This is \cite[Thm.~5.9]{rosen2013number}.

\item This is \cite[Ch.~VII, Thm.~(10.4)(iii)]{Neukirch2006}.

\item This is \cite[Ch.~VII, Cor.~(10.5)(iii)]{Neukirch2006}.
\end{enumerate}
Note that the proofs in \cite{Neukirch2006} don't use the fact that the $L$-functions considered there are defined over number fields.
\end{proof}

\begin{notat}
Let $L^*_S(\chi,0)$ be the leading term of the Taylor expansion of $L_S(\chi,s)$ at $s=0$.
\end{notat}

Now suppose that $L$ is a \emph{real} abelian extension of $K$. Define $S_\infty(L)$ to be the set of all primes of $L$ lying over $\infty$ (if the extension is clear, we will simply write $S_\infty$). Since $L/K$ is real, these are exactly $[L:K]$ many primes and each has norm $N\infty=q^{d_\infty}$. Note that $\mc O_L^\times/\mu(L)$ is a free $\bb Z$-module of rank $\abs{S_\infty(L)}-1$ and hence we can choose units $u_1,...,u_{[L:K]-1}$ which project to a basis. Choosing a place $w_0\in S_\infty(L)$, we can define a matrix 
\begin{align*}
(-d_\infty \ord_w(u_i))_{\substack{w\in S_\infty(L)\setminus \{w_0\}\\ i\in \{1,...,[L:K]-1\}}}\in \bb Z^{([L:K]-1)\times ([L:K]-1)}.
\end{align*}
Then we define the \emph{regulator} $R_L$ of $L$ as the absolute value of the determinant of this matrix. Note that the regulator $R_L^{Ros}$ defined in \cite[Ch.~14]{rosen2013number} can be obtained from our definition by 
\begin{align}\label{eq:regs}
R_L^{Ros}=(\log(q))^{[L:K]-1}R_L.
\end{align}
Hence we obtain the

\begin{thm}[Analytic class number formula]\label{thm:analyticCNF}
We have
\begin{align*}
\zeta_{L,S_\infty}^*(0) &=-(\log (q))^{[L:K]-1}\frac{h_LR_L}{w_L}.
\end{align*}
\end{thm}

\begin{proof}
This is \cite[Thm.~14.4]{rosen2013number} together with \eqref{eq:regs}.
\end{proof}

\begin{kor}
We have
\begin{align*}
\zeta_{K,\{\infty\}}^*(0)=-\frac{h}{w_K}
\end{align*}
\end{kor}

\begin{proof}
Since $\mc O_K^\times=\mu(K)$, we get $R_K=1$.
\end{proof}


\subsection{Kronecker's Limit Formulae}\label{sec:kronecker}

We fix a prime $w_0\in S_\infty(H_\mf m)$. Then for each subfield $L$ of $H_\mf m$, there is a unique prime in $S_\infty(L)$ below $w_0$. Since $\infty$ splits completely in $H_\mf m$, the valuations of these primes are compatible. By abuse of notation, we denote each of these valuations by $\ord_\infty$, i.e. for an element $x\in H_\mf m$ we implicitly set
\begin{align*}
\ord_\infty(x):=\ord_{w_0}(x)
\end{align*}
and analogously for each subfield $L$ of $H_\mf m$. The same convention will be used for absolute values.\\

For $\mf n\neq (1)$ let $S_\mf n:=\{\mf p\ssq \mc O_K\mid \mf p \text{ prime},\mf p\mid \mf n\}$ be the support of $\mf n$. Now we can state Kronecker's second limit formula:

\begin{prop}\label{prop:starkUnit}
\begin{enumerate}[label=(\roman*)]
\item Let $(1)\neq \mf n\mid \mf m$ and let $\chi\in \widehat{\Gal(H_\mf n/K)}$. Then we have
\begin{align*}
L_{S_\mf n}(\chi,0)=\frac{1}{w_\infty}\sum_{\sigma\in \Gal(H_\mf n/K)}\ord_{\infty}(\alpha_\mf n^\sigma)\chi(\sigma).
\end{align*}

\item For any non-trivial character $\chi\in \widehat{\Gal(H/K)}$, we have
\begin{align*}
L(\chi,0)=\frac{1}{w_\infty h}\sum_{\sigma\in \Gal(H/K)}\ord_{\infty}(\partial(\sigma))\chi(\sigma).
\end{align*}

\end{enumerate}
\end{prop}

\begin{proof}
Part (i) is exactly the last equation in \cite{hayes1985stickelberger}, whereas part (ii) follows directly from \cite[Proof of Prop.~3]{oukhaba1997groups} and Remark \ref{rem:w_K}.
\end{proof}

\begin{bem}
By the proposition above, we can regard the ramified elliptic units as Stark units. Indeed, if $\mf n\neq (1)$ then the set $S:=S_\mf n\cup \{\infty\}$ contains all places which ramify in $H_\mf n/K$ and $\abs{S}\geq 2$. Moreover, $S$ contains the completely split prime $\infty$. Then Stark's conjecture (cf.~\cite[Ch.~IV, Conj.~2.2]{tate1984conjectures}) predicts the existence of an element $\varepsilon$ such that
\begin{align*}
L'_S(\chi,0)=-\frac{1}{w_{H_\mf n}}\sum_{\sigma\in \Gal(H_\mf n/K)} \log\abs{\varepsilon^\sigma}_{\infty} \chi(\sigma)
\end{align*}
for all $\chi\in \widehat{\Gal(H_\mf n/K)}$. By definition of the $L$-function, we obtain
\begin{align*}
L_{S}(\chi,s)=(1-\chi(\infty)N\infty^{-s})L_{S_\mf n}(\chi,s)=(1-N\infty^{-s})L_{S_\mf n}(\chi,s)
\end{align*}
and hence
\begin{align*}
L'_S(\chi,0)&=\log(N\infty)L_{S_\mf n}(\chi,0)=-\frac{1}{w_\infty}\sum_{\sigma\in \Gal(H_\mf n/K)}\log\lrk N\infty^{-\ord_{\infty} (\alpha_\mf n^\sigma)}\rrk \chi(\sigma)\\
&=-\frac{1}{w_\infty}\sum_{\sigma\in \Gal(H_\mf n/K)} \log\abs{\alpha_\mf n^\sigma}_{\infty} \chi(\sigma).
\end{align*}
\end{bem}


\section{Sinnott's module}

Let $L/K$ be a fixed finite real abelian extension of conductor $\mf m$ (as in Section \ref{sec:groupell}). Remember that for a prime $\mf p$ of $K$ the element $\sigma_\mf p\in G=\Gal(L/K)$ is the lift of an associated Frobenius element in $D_\mf p/T_\mf p$. Define $\tau_\mf p:=\sigma_\mf p^{-1} e_{T_\mf p}\in \bb Q[G]$.

\begin{defn}
\begin{enumerate}[label=(\roman*)]
\item We define $\rho'_\mf n:=s(\Gal(L/L_\mf n))\prod_{\mf p\mid \mf n} (1-\tau_\mf p)$ for any integral ideal $\mf n$, where the product runs over all prime ideals dividing $\mf n$. 

\item The $\bb Z[G]$-submodule $U'$ of $\bb Q[G]$ generated by $\rho'_\mf n$ where $\mf n$ runs through all integral ideals of $\mc O_K$ is called \emph{Sinnott's module}.

\item Define $U'_0$ to be the kernel of multiplication by $s(G)$ in $U'$.
\end{enumerate}
\end{defn}

\begin{bems}\label{rem:gensinnott}
\begin{enumerate}[label=(\roman*)]
\item The notation $U'$ and $\rho'_\mf n$ is adopted from \cite{chapdelaine2017annihilators}. In the second part of this article, we use a modification of Sinnott's module which will be denoted by $U$.

\item Note that for $\mf n\nmid \mf m$, we have $L_\mf n=L_{(1)}$, hence $\rho'_\mf n=\rho'_{(1)}$. Therefore, it suffices to consider the elements $\rho'_\mf n$ with $\mf n\mid \mf m$.

\item If $\mf n\neq (1)$, we have $\rho'_\mf n\in U'_0$. As in the imaginary quadratic case (cf.~\cite{oukhaba2003index}) the component of $U'$ generated by $\rho'_{(1)}$ intersected with $U'_0$ is generated by 
\begin{align*}
\rho'_{(1)}(1-\sigma),\quad \sigma\in G.
\end{align*}
If $\sigma,\sigma'\in G$ are lifts of the same element $\tau\in \Gal(L_{(1)}/K)$, then 
\begin{align*}
\rho'_{(1)}(1-\sigma)=\rho'_{(1)}(1-\sigma'),
\end{align*}
hence it suffices to consider the elements 
\begin{align*}
\rho'_{(1)}(1-\widetilde \tau), \quad \tau\in \Gal(L_{(1)}/K),
\end{align*}
where $\widetilde \tau\in G$ is an arbitrary lift of $\tau$.

\end{enumerate}
\end{bems}

Now remember the convention introduced in Section \ref{sec:kronecker} and consider the logarithmic map 
\begin{align*}
l_L:L^\times&\lra \bb Q[G]\\
x&\lmt \sum_{\sigma\in G} \ord_\infty(x^\sigma)\sigma^{-1}
\end{align*}
and the element
\begin{align*}
\omega:=hw_\infty \sum_{\substack{\chi\in \widehat{G}\\ \chi\neq 1}} L(\overline{\chi},0)e_\chi\ .
\end{align*}
Also define 
\begin{align*}
l_L^*:=(1-e_G)l_L\ .
\end{align*}

\begin{prop}[{cf.~\cite[Prop.~6]{oukhaba2003index}}]\label{prop:imglog}
Let $\mf n\neq 1$ be such that $\mf n\mid \mf m$ and let \linebreak $\tau\in \Gal(L_{(1)}/K)$, then
\begin{align*}
l_L^*(\varphi_{L,\mf n})&=\omega \rho'_\mf n\ ,\\
l_L^*\lrk  \frac{\partial_L(1)}{\partial_L(\tau)}\rrk &=\omega \rho'_{(1)} (1-\widetilde \tau)\ ,
\end{align*}
where $\widetilde\tau\in G$ is any lift of $\tau$.
\end{prop}

\begin{proof}
It suffices to prove the equations on the $\chi$-component for each non-trivial character $\chi\in \widehat{G}$. This follows from Proposition \ref{prop:starkUnit} by a direct computation.
\end{proof}

\begin{kor}\label{kor:imagelognumbers}
We have $l^*_L(P_L)=\omega\cdot U'_0$.
\end{kor}

\begin{proof}
This follows directly from Remark \ref{rem:gensinnott}.
\end{proof}

\section{An index formula}\label{sec:oukhabaend}

We briefly recall the definition of Sinnott's Index (see \cite[§4]{oukhaba2003index}). Let $V$ be a finite-dimensional vector space over $L=\bb Q$ or $\bb R$. A subgroup $X$ of $V$ is a \emph{lattice} if $\rk_\bb Z(X)=\dim_L(V)$ and $L\cdot X=V$. If $A$ and $B$ are lattices of $V$ and $\gamma$ is an automorphism of $V$ such that $\gamma(A)=B$, then we define
\begin{align*}
[A:B]:=\abs{\det(\gamma)}.
\end{align*}
If $B\subseteq A$, then $[A:B]$ is the usual group index. Now we can prove

\begin{prop}[{cf.~\cite[Prop.~7]{oukhaba2003index}}]\label{prop:indexsinnott}
We have
\begin{align*}
[U'_0:l_L^*(P_L)]=\lrk \frac{hw_\infty}{d_\infty}\rrk^{[L:K]-1}\cdot \frac{w_K h_L R_L}{w_L h}
\end{align*}
\end{prop}

\begin{proof}
Using Proposition \ref{prop:Lfunctions}(iv) and
\begin{align*}
L'_{\{\infty\}}(\chi,0)=\log (N\infty)\cdot L(\chi,0)= d_\infty \log (q)\cdot L(\chi,0), 
\end{align*}
we obtain
\begin{align*}
\zeta^*_{L,S_\infty}(0)&=\zeta^*_{K,\{\infty\}}(0)\cdot \prod_{\chi\neq 1}L'_{\{\infty\}}(\chi,0)\\
&=- \frac{h}{w_K}(d_\infty\log (q))^{[L:K]-1}\prod_{\chi\neq 1} L(\chi,0).
\end{align*}
Together with the analytic class number formula \ref{thm:analyticCNF}, this yields
\begin{align*}
\prod_{\chi\neq 1} L(\chi,0)= \frac{w_Kh_LR_L}{w_L h d_\infty^{[L:K]-1}}.
\end{align*}

Therefore, we obtain
\begin{align*}
\abs{\det(\omega)}&=\prod_{\chi\neq 1} \chi(\omega)=(hw_\infty)^{[L:K]-1}\prod_{\chi\neq 1} L(\overline{\chi},0)\\
&=\lrk \frac{hw_\infty}{d_\infty}\rrk^{[L:K]-1}\cdot \frac{w_K h_L R_L}{w_L h},
\end{align*}
where the first equality follows from \cite[Lemma 1.2 (b)]{sinnott1980stickelberger}. Since this is non-zero, we find that multiplication by $\omega$ is an automorphism of $V=\bb Q\cdot U'_0$. By Corollary \ref{kor:imagelognumbers} we have $l_L^*(P_L)=\omega U'_0$ and hence the desired Sinnott index exists and is given by
\begin{align*}
[U'_0:l_L^*(P_L)]=[U'_0:\omega U'_0]=\abs{\det(\omega)}=\lrk \frac{hw_\infty}{d_\infty}\rrk^{[L:K]-1}\cdot \frac{w_K h_L R_L}{w_L h}.
\end{align*}
\end{proof}

Let $\mf p\mid \mf m$ be a prime ideal of $K$. The norm relation of Proposition \ref{prop:normrel} implies that $x_\mf p^{w_\infty/w_K[H:L_{(1)}]}\in P_L$, where $x_\mf p$ is a generator of $\mf p^h$.

\begin{defn}\label{defn:QL}
Let $Q_L$ be the subgroup of $P_L$ generated by $\mu(L), \Delta_L$ and the elements $x_\mf p^{w_\infty/w_K[H:L_{(1)}]}$ for all $\mf p\mid \mf m$.
\end{defn}

Now we can state the next

\begin{prop}[{cf.~\cite[Prop.~8]{oukhaba2003index}}]\label{prop:indexlog}
We have
\begin{align*}
[l_L^*(P_L):l_L(C_L)]=\frac{\prod_\mf p [L\cap H_{\mf p^\infty}:L_{(1)}]}{[P_L^{w_L}\cap K:Q_L^{w_L}\cap K]}
\end{align*}
where $\mf p$ runs through all maximal ideals of $\mc O_K$. 
\end{prop}

\begin{proof}
We can use the proof of \cite[Prop.~8]{oukhaba2003index} here. The essential inputs 
\begin{enumerate}[label=(\roman*)]
\item $\ker(l_L)\cap \mc O_L^\times=\mu(L)$,

\item $l_L(C_L)=l_L^*(C_L)$,

\end{enumerate}
also hold in the function field case.
\end{proof}

Now we can state the desired index formula (cf.~Theorem A):

\begin{thm}[{cf.~\cite[Thm.~1]{oukhaba2003index}}]\label{thm:indexformula}
Set $d(L):=[P_L^{w_L}\cap K:Q_L^{w_L}\cap K]$. Then we get
\begin{align*}
[\mc O_L^\times:C_L]=\frac{(hw_\infty)^{[L:K]-1}w_Kh_L}{w_Lh}\frac{\prod_\mf p [L\cap H_{\mf p^\infty}:L_{(1)}]}{[L:L_{(1)}]}\frac{[\bb Z[G]:U']}{d(L)}\ .
\end{align*}
\end{thm}

\begin{proof}
Let $R=\bb Z[G]$ and $R_0$ be the kernel of multiplication with $s(G)$ in $R$. Since $\ker(l_L)\cap \mc O_L^\times=\mu(L)$ we get
\begin{align*}
[\mc O_L^\times:C_L]&=[l_L(\mc O_L^\times):l_L(C_L)]=[l_L(\mc O_L^\times):R_0][R_0:l_L(C_L)]\\
&=\frac{[R_0:U'_0]}{[R_0:l_L(\mc O_L^\times)]}[U'_0:l_L(C_L)]\\
&=\frac{[R_0:U'_0]}{[R_0:l_L(\mc O_L^\times)]}[U'_0:l_L^*(P_L)][l_L^*(P_L):l_L(C_L)].
\end{align*}
Note that all the indices above are defined, since each of the $\bb Z$-modules has the same rank. By definition of Sinnott's index, one can easily show that
\begin{align*}
[R_0:l_L(\mc O_L^\times)]=\abs{\det(A)},
\end{align*}
where $A$ is the matrix with entries
\begin{align*}
(\ord_w(u_i))_{\substack{w\in S_\infty(L)\setminus\{w_0\}\\i\in \{1,...,[L:K]-1\}}},
\end{align*}
where $w_0$ is an arbitrary place in $S_\infty(L)$ and the units $u_1,...,u_{[L:K]-1}\in \mc O_L^\times$ project to a basis of $\mc O_L^\times/\mu(L)$. By the definition of the regulator, we hence get
\begin{align*}
R_L=\abs{\det(-d_\infty A)}=d_\infty^{[L:K]-1}\abs{\det(A)},
\end{align*}
so
\begin{align*}
[R_0:l_L(\mc O_L^\times)]=\frac{R_L}{d_\infty^{[L:K]-1}}.
\end{align*}
As in \cite{oukhaba2003index} we get that
\begin{align*}
[R_0:U'_0]=\frac{[R:U']}{[L:L_{(1)}]}\ .
\end{align*}
Using these computations and the results of the Propositions \ref{prop:indexsinnott} and \ref{prop:indexlog} we obtain
\begin{align*}
[\mc O_L^\times:C_L]=\frac{(hw_\infty)^{[L:K]-1}w_Kh_L}{w_Lh}\frac{\prod_\mf p [L\cap H_{\mf p^\infty}:L_{(1)}]}{[L:L_{(1)}]}\frac{[R:U']}{d(L)}\ .
\end{align*}

\end{proof}

We state some results on $[R:U']$ similar to \cite[§6,§7]{oukhaba2003index}:

\begin{prop}\label{prop:RU}
\begin{enumerate}[label=(\roman*)]
\item The index $[R:U']$ is an integer divisible only by primes dividing $[L:L_{(1)}]$. Moreover, if $\Gal(L/L_{(1)})$ is the direct product of its inertia groups or if at most two primes ramify in $L/K$ then $[R:U']=1$.

\item If $G$ is cyclic, then $[R:U']=1$.

\item If $L=H_\mf m$ for some integral ideal $\mf m=\prod_{i=1}^s \mf p_i^{e_i}$ for some $s\geq 3$ and $h$ is coprime to $w_K$, we get
\begin{align*}
[R:U']=w_K^{e(2^{s-2}-1)},
\end{align*}
where $e$ is the index of the subgroup generated by the classes of $\mf p_i$ in $Cl(K)$.
\end{enumerate}
\end{prop}

\begin{proof}
\begin{enumerate}[label=(\roman*)]
\item This is \cite[Prop.~16]{oukhaba2003index}.

\item This is \cite[Thm.~5.3]{sinnott1980stickelberger}.

\item This is \cite[Prop.~18]{oukhaba2003index}.
\end{enumerate}
Note that the arguments are only based on the group structure of $G$ and hence can also be applied in the case of function fields.
\end{proof}

\begin{bem}\label{rem:dL}
There is a list of cases in \cite[Remark 2]{oukhaba2003index} in which he gets $d(L)=1$. With similar methods we can show that if one of the following conditions holds, we have $d(L)=1$:
\begin{enumerate}[label=(\roman*)]
\item $L\ssq H$,
\item $H\ssq L$,
\item $[H:L_{(1)}]$ and $[L:L_{(1)}]$ are coprime,
\item $\Gal(L/L_{(1)})$ is cyclic,
\item $\Gal(L/L_{(1)})$ is the direct product of its inertia subgroups,
\item at most two primes ramify in $L/K$.
\end{enumerate}
\end{bem}

\begin{bem}
\begin{enumerate}[label=(\roman*)]
\item In \cite{oukhaba1992groups} H. Oukhaba defines a group $\mc E_L$ of elliptic units in an unramified extension $L/K$. He also shows that the elements of $\mc E_L^{w_Kw_\infty h}$ are of the form
\begin{align*}
\prod_{\tau\in G} \lrk\frac{\partial_L(1)\partial_L(\tau\sigma^{-1})}{\partial_L(\sigma^{-1})\partial_L(\tau)}\rrk^{w_Km_\tau}
\end{align*}
for $\sigma\in G$ and certain rational numbers $m_\tau\in \bb Q$ (cf.~Prop.~3.6 in loc.~cit.). He also derives an index formula in this case:
\begin{align*}
[\mc O_L^\times:\mc E_L]=\frac{h_L}{[H:L]}.
\end{align*}
In this case, our index formula yields
\begin{align*}
[\mc O_L^\times:C_L]=(hw_\infty)^{[L:K]-1}\frac{w_K h_L}{w_Lh}.
\end{align*}
From the above description, we find that $\mc E_L^{w_\infty h}\ssq C_L$ and we get
\begin{align*}
[C_L:\mc E_L^{w_\infty h}]=h\frac{w_L}{w_K}.
\end{align*}

\item In \cite{yin1997index} L. Yin defines a group $\overline{C}$ of extended cyclotomic units in the ray class fields $K_\mf m$. The ramified elliptic units in this article are in fact norms of Yin's cyclotomic units. However our construction of the unramified units is quite different to the one in \cite{yin1997index}. Nevertheless, Yin also computes an index formula
\begin{align*}
[\mc O_{H_\mf m}^\times:(\overline{C}\cap \mc O_{H_\mf m}^\times)]=w_K^ah_{H_\mf m},
\end{align*}
where $a=0$ if $s\leq 2$ and $a=e(2^{s-2}-1)-(s-2)$ if $s\geq 3$. Note that there is the additional assumption $(h,w_K)=1$ in the case $s\geq 3$. With these assumptions, we get from our index formula
\begin{align*}
[\mc O_{H_\mf m}^\times:C_{H_\mf m}]=(hw_\infty)^{[H_\mf m:K]-1}\frac{w_K h_{H_\mf m}}{w_\infty h}w_K^{-(s-1)}[R:U'].
\end{align*}
With Proposition \ref{prop:RU}, this yields
\begin{align*}
[\mc O_{H_\mf m}^\times:C_{H_\mf m}]=(hw_\infty)^{[H_\mf m:K]-2}w_K^ah_{H_{\mf m}}.
\end{align*}
\end{enumerate}
\end{bem}

\section{A nontrivial root of an elliptic unit}\label{sec:CK}

With this definition of elliptic units we can prove an analogue of the main result of \cite{chapdelaine2017annihilators} in the case of global function fields.

\subsection{Preliminaries}

We use the notation from Section \ref{sec:notat1} with the following additional assumptions:
\begin{itemize}
\item Suppose $p$ is an odd prime such that $p\nmid q\cdot(q-1)\cdot h$. 

\item $L$ is a real cyclic extension of $K$ of degree $p^k$ for some positive integer $k$.

\item We change the notation to $\Gamma:=\Gal(L/K)$. Let $\sigma$ be a generator of $\Gamma$.
\end{itemize}

\begin{bem}
Note that the assumption on $L$ and $p\nmid h$ are exactly the same as in \cite{chapdelaine2017annihilators}. The assumption $p\nmid (q-1)=w_K$ is also implied by the hypotheses stated there. The only new assumption is $p\nmid q$, i.e. we suppose that $p$ is prime to the characteristic of $K$, which is a natural hypothesis when dealing with function fields.
\end{bem}

Note that since $p\nmid h$, we have 
\begin{align*}
L\cap H=K
\end{align*}
and we assume that there are exactly $s\geq 2$ primes $\mf p_1,...,\mf p_s$ of $K$ which ramify in $L$. Set
\begin{itemize}
\item $I:=\{1,...,s\}$,

\item $x_j:=x_{\mf p_j}$ is a generator of $\mf p_j^h$,

\item $\mf P_j$ is a fixed prime ideal of $L$ over $\mf p_j$,

\item For any abelian extension $M/K$ let $D_j(M):=D_{\mf p_j}\ssq\Gal(M/K)$ be the decomposition group of $\mf p_j$ and $T_j(M):=T_{\mf p_j}\ssq D_j(M)$ be the inertia group of $\mf p_j$,

\item $t_j:=\abs{T_j(L)}$ is the ramification index of $\mf P_j$ over $\mf p_j$,

\item $n_j:=[G:D_j(L)]$.
\end{itemize}

Then it follows that $t_jn_j\mid p^k$ and 
\begin{align*}
\mf p_j\mc O_L=\prod_{i=0}^{n_j-1} \mf P_j^{t_j\sigma^i}\ .
\end{align*}
Since $p\nmid q$, this implies that the extension $L/K$ is tamely ramified and hence its conductor is square-free. Therefore the conductor is given by $\mf m:=\mf m_I:=\prod_{j\in I} \mf p_j$. 

\subsection{The distinguished subfields $F_j$}

For any subset $\emptyset\neq J\ssq I$ we set $\mf m_J:=\prod_{j\in J} \mf p_j$. With our previous observation we find that $L\ssq H_{\mf m_I}$. 

\begin{lem}
$L\ssq \prod_{j\in I} H_{\mf p_j}$.
\end{lem}

\begin{proof}
By class field theory, we have a canonical isomorphism
\begin{align*}
\Gal(H_{\mf m}/H)\cong (\mc O_K/\mf m)^\times /\im (\mu(K)),
\end{align*}
(see e.g.~\cite[Eq.~(3.1)]{hayes1985stickelberger}). With the Chinese Remainder Theorem, we get
\begin{align*}
[H_{\mf m}:\prod_{j\in I}H_{\mf p_j}]&= \frac{[H_\mf m:H]}{[\prod_{j\in I} H_{\mf p_j}:H]}=\frac{\abs{(\mc O_K/\mf m)^\times}/w_K}{\prod_{j\in I}[H_{\mf p_j}:H]}=\frac{\prod_{j\in I}\abs{(\mc O_K/\mf p_j)^\times}}{w_K\prod_{j\in I}\abs{(\mc O_K/\mf p_j)^\times}/w_K}\\
&=w_K^{s-1}.
\end{align*}
The second equality follows, since for any $2\leq j\leq s$ we obtain $H_{\mf p_j}\cap \prod_{i=1}^{j-1} H_{\mf p_i}=H$ by considering the ramification of $\mf p_j$. Since $p\nmid w_K$, we get $L\ssq \prod_{j\in I} H_{\mf p_j}$.
\end{proof}

Using the canonical isomorphism of the proof above, we obtain $$\Gal(H_{\mf p_j}/H)\cong (\mc O_K/\mf p_j)^\times /\im (\mu(K))$$ which is a cyclic group. Since $t_j\mid [L_{\mf p_j}:K]\mid [H_{\mf p_j}:K]$ and $p\nmid h$, it follows that $t_j\mid [H_{\mf p_j}:H]$. Using $p\nmid h$ and \cite[Lemma 2.1]{chapdelaine2017annihilators} we can define $F_j$ to be the unique subfield of $H_{\mf p_j}$ such that $[F_j:K]=t_j$. Then $F_j\cap H=K$ and $F_j/K$ is totally ramified at $\mf p_j$ and unramified everywhere else.
\\

From now on, we will write $H_J:=H_{\mf m_J}$ for each $\emptyset\neq J\ssq I$ and $$F_J:=\prod_{j\in J} F_j\ssq H_J.$$ Note that the conductor of $F_J$ is $\mf m_J$. The definition of $F_I$ implies that the Galois group $\Gal(F_I/F_{I\setminus \{j\}})=T_j(F_I)$ is the inertia subgroup of a prime of $F_I$ above $\mf p_j$, in particular for each $j\in I$ we have $\abs{\Gal(F_I/F_{I\setminus \{j\}})}=t_j$. 

\begin{lem}\label{lem:subset}
For any two subsets $\emptyset\neq J_1\ssq J_2\ssq I$, we have $F_{J_1}=F_{J_2}\cap H_{J_1}$. Moreover, $F_I\cap H=K$.
\end{lem}

\begin{proof}
The inclusion $F_{J_1}\ssq F_{J_2}\cap H_{J_1}$ is clear. For the other inclusion, we use induction on $n=\abs{J_2\setminus J_1}$. The case $n=0$, i.e. $J_1=J_2$, is clear. If $n\geq 1$, we fix an index $j\in J_2\setminus J_1$ and we see that
\begin{align*}
F_{J_2}\cap H_{J_1}\ssq F_{J_2}\cap H_{J_2\setminus\{j\}}\ssq F_{J_2\setminus\{j\}}
\end{align*}
by the induction hypothesis. But we clearly also have $F_{J_2}\cap H_{J_1}\ssq H_{J_1}$, hence
\begin{align*}
F_{J_2}\cap H_{J_1}\ssq F_{J_2\setminus\{j\}}\cap H_{J_1}\ssq F_{J_1}
\end{align*}
by the induction hypothesis.\\

The second assertion follows, since $[F_I:K]$ is a $p$-power and $p\nmid h$.
\end{proof}

\begin{prop}[{cf.~\cite[Prop.~2.2]{chapdelaine2017annihilators}}]\label{prop:Gstructure}
We have $F_jH_{I\setminus \{j\}}=LH_{I\setminus \{j\}}$ for each $j\in I$. The Galois group
\begin{align*}
G=\Gal(F_I/K)=\prod_{j\in I} \Gal(F_I/F_{I\setminus \{j\}})
\end{align*}
is the direct product of its inertia subgroups. Moreover $L\ssq F_I$.
\end{prop}

\begin{proof}

We can take the proof of \cite[Prop.~2.2]{chapdelaine2017annihilators} here, since there are no changes necessary.
\end{proof}

\begin{kor}[{cf.~\cite[Cor.~2.3]{chapdelaine2017annihilators}}]\label{cor:1}
\begin{enumerate}[label=(\roman*)]
\item For each $j\in I$ we have
\begin{align*}
T_j(L)=\Gal(L/L\cap F_{I\setminus \{j\}})=\langle \sigma^{p^k/t_j}\rangle\ .
\end{align*}
Moreover, $F_{I\setminus \{j\}} L=F_I$ and $[L\cap F_{I\setminus \{j\}}:K]=\frac{p^k}{t_j}$.

\item $F_I/L$ is an unramified abelian extension.

\item There exists at least one index $j_0\in I$ such that $t_{j_0}=p^k$ so that $G=\Gal(F_I/K)$ has exponent $p^k$.

\end{enumerate}
\end{kor}

\begin{proof}
See \cite[Cor.~2.3]{chapdelaine2017annihilators}.
\end{proof}

\subsection{The elliptic units}

Since $F_{I}\cap H=K$ by Lemma \ref{lem:subset}, there are no unramified elliptic units and we define
\begin{align*}
\eta_J:=N_{H_J/F_J}(\alpha_{\mf m_J})=\varphi_{F_I,\mf m_J}^{1/h}\in \mc O_{F_J},
\end{align*}
cf.~Remark \ref{rem:hpower}. Let $\sigma_j\in G=\Gal(F_I/K)$ be the lift of the Frobenius of $\mf p_j$ uniquely defined by $\sigma_j\vert_{F_{I\setminus \{j\}}}=(\mf p_j, F_{I\setminus \{j\}}/K)$ and $\sigma_j\vert_{F_j}=1$. Then we can state the next 

\begin{lem}[{cf.~\cite[Lemma 3.1]{chapdelaine2017annihilators}}]\label{lem:4}
For any $j\in I$ we have $$D_j(L)=\langle \sigma^{n_j}\rangle =\langle \sigma_{j}\vert_L, \sigma^{p^k/t_j}\rangle\ .$$
\end{lem}

\begin{proof}
See \cite[Lemma 3.1]{chapdelaine2017annihilators}.
\end{proof}

\begin{lem}[{cf.~\cite[Lemma 3.2]{chapdelaine2017annihilators}}]\label{lem:rootsofunity}
We have $\mu(F_I)=\mu(K)$.
\end{lem}

\begin{proof}
For $\zeta\in \mu(F_I)$, the extension $K(\zeta)/K$ is a constant field extension. Since all constant field extensions are unramified, we obtain $K(\zeta)\ssq F_I\cap H=K$, so $\zeta\in \mu(K)$.
\end{proof}

Proposition \ref{prop:normrel} implies that for each $J\ssq I$ and each $j\in J$ 
\begin{align}
N_{F_J/F_{J\setminus \{j\}}}(\eta_J)=\begin{cases} \eta_{J\setminus \{j\}}^{1-\sigma_j^{-1}},\quad &J\setminus\{j\}\neq \emptyset,\\
\zeta x_j^{w_\infty/w_K} & J\setminus\{j\}= \emptyset,
\end{cases} \label{eq:normrel1}
\end{align} 
for some $\zeta\in \mu(K)$.

In analogy to \cite{chapdelaine2017annihilators}, we use the following definition of elliptic units:

\begin{defn}
\begin{itemize}
\item The \emph{group of elliptic numbers $\mc P_{F_I}$ of $F_I$} is defined to be the $\bb Z[G]$-submodule of $F_I^\times$ generated by $\mu(K)$ and by $\eta_J$ for all $\emptyset\neq J\ssq I$.

\item The \emph{group of elliptic units $\mc C_{F_I}$ of $F_I$} is then defined as $\mc C_{F_I}:=\mc P_{F_I}\cap \mc O_{F_I}^\times$. 

\item The \emph{group of elliptic numbers $\mc P_L$ of $L$} is defined as the $\bb Z[\Gamma]$-submodule of $L^\times$ generated by $\mu(K)$ and $N_{F_J/F_J\cap L}(\eta_J)$ for all $\emptyset\neq J\ssq I$.

\item The \emph{group of elliptic units $\mc C_L$ of $L$} is defined as $\mc C_L:=\mc P_L\cap \mc O_L^\times$.
\end{itemize}
\end{defn}

Since $F_I\cap H=K=L\cap H$, one can check that these elliptic units are related to the units of Definition \ref{defn:ellipticunits} by
\begin{align*}
C_{F_I}&=\mc C_{F_I}^h\cdot \mu(K)\ ,\\
C_{L}&= \mc C_{L}^h\cdot \mu(K)\ .
\end{align*}

This fact and Theorem \ref{thm:indexformula} imply the next Lemma. We first need the following

\begin{notat}
Let $\widetilde L$ be the maximal subfield of $L$ containing $K$ such that $\widetilde L/K$ is ramified in at most one prime ideal of $K$.
\end{notat}

Note that since $\Gamma$ is cyclic and of prime power order, the field $\widetilde L$ is unique.

\begin{lem}[{cf.~\cite[Lemma 3.4]{chapdelaine2017annihilators}}]\label{lem:indexunit}
\begin{enumerate}[label=(\roman*)]
\item We have
\begin{align*}
[\mc O_{F_I}^\times:\mc C_{F_I}]&=w_\infty^{[F_I:K]-1}\frac{h_{F_I}}{h}\ ,\\
[\mc O_{L}^\times:\mc C_{L}]&=w_\infty^{[L:K]-1}\frac{h_{L}}{h[L:\widetilde L]}\ .
\end{align*}

\item For $\beta\in \mc P_{F_I}$ we have $\beta\in \mc C_{F_I}$ if and only if $N_{F_I/K}(\beta)\in \mu(K)$.
\end{enumerate}
\end{lem}

\begin{proof}[{Sketch of a proof}]
For more details and part (ii) see \cite[Lemma 3.4]{chapdelaine2017annihilators}. First, we get from the above observation $[\mc C_{F_I}:C_{F_I}]=h^{[F_I:K]-1}$ and $[\mc C_{L}:C_{L}]=h^{[L:K]-1}$. Moreover, it follows from Proposition \ref{prop:Gstructure}, Proposition \ref{prop:RU} (ii) (resp.~(i)) and Remark~\ref{rem:dL}(iv) (resp.~(v)) that the last quotient in \ref{thm:indexformula} is equal to $1$ for $L$ (resp.~$F_I$). We also obtain $w_{F_I}=w_L=w_K$ by Lemma \ref{lem:rootsofunity}, $L_{(1)}=K$, $F_I\cap H_{\mf p^\infty}=F_{j}$ for $\mf p=\mf p_j$ and $\prod_{j=1}^s [F_{j}:K]=[F_I:K]$, which yields the first equation. For the second equation, we consider the definition of $\widetilde L$. By part (iii) of Corollary \ref{cor:1} we know that there is at least one prime $\mf p_i$ which is totally ramified in $L$. Therefore, $\widetilde L$ is the maximal subfield of $L$ which is unramified at every prime except $\mf p_i$, hence $\widetilde L=L\cap H_{\mf p_i^\infty}$. Since for $\mf p\neq \mf p_i$ the extension $L\cap H_{\mf p^\infty}$ is unramified at $\mf p_i$ and $\mf p_i$ is totally ramified in $L$, we find that $L\cap H_{\mf p^\infty}=K$ for $\mf p\neq \mf p_i$, hence we obtain $\prod_\mf p [L\cap H_{\mf p^\infty}:K]=[\widetilde L:K]$.
\end{proof}

Now we use a modification of Sinnott's module defined in \cite{greither2014linear}. This module $U$ is a $\bb Z[G]$-submodule of $\bb Q[G]\oplus \bb Z^s$ generated over $\bb Z[G]$ by certain elements $\rho_J$, $J\ssq I$. Each $\bb Z$ summand is endowed with the trivial $G$-action and has a standard basis element denoted by $e_j$. 

Define 
\begin{align*}
\Psi:\mc P_{F_I}&\lra U\\
\eta_J&\lmt\rho_{I\setminus J}
\end{align*}
for $\emptyset\neq J\ssq I$ and $\Psi(\mu(K))=0$.

\begin{lem}[{cf.~\cite[Lemma 3.5]{chapdelaine2017annihilators}}]\label{lem:psi}
$\Psi$ is a well-defined $\bb Z[G]$-module homomorphism satisfying $\ker(\Psi)=\mu(K)$ and $U=\Psi(\mc P_{F_I})\oplus (s(G)\bb Z)$.
\end{lem}

\begin{proof}
See \cite[Lemma 3.5]{chapdelaine2017annihilators}. 
\end{proof}

We call
\begin{align*}
\eta:=N_{F_I/L}(\eta_I)
\end{align*}
the \emph{top generator} of both $\mc P_L$ and $\mc C_L$. Set $B:=\Gal(F_I/L)\ssq \Gal(F_I/K)=G$, then we have $\Gamma=\langle \sigma\rangle\cong G/B$.

\begin{lem}[{cf.~\cite[Lemma 4.1]{chapdelaine2017annihilators}}]\label{lem:9}
An elliptic number $\beta\in \mc P_{F_I}$ belongs to $L$ if and only if $\Psi(\beta)$ is fixed by $B$, i.e. $\Psi(\mc P_{F_I})^B=\Psi(\mc P_{F_I}\cap L)$.
\end{lem}

\begin{proof}
See \cite[Lemma 4.1]{chapdelaine2017annihilators}.
\end{proof}

Recall that $n_i$ is the index of the decomposition group of the ideal $\mf P_i\ssq L$ in $\Gamma$. Without loss of generality we can assume 
\begin{align*}
n_1\leq n_2\leq\cdots \leq n_s
\end{align*}
and set $n=n_s=\max\{n_i\mid i\in I\}$. Since $p\mid t_s$ we have $n\mid p^{k-1}$ and by Corollary \ref{cor:1}(iii) we get $t_1=p^k$ and hence $n_1=1$. Let $L'$ be the unique subfield of $L$ containing $K$ such that $[L':K]=n$. Note that $\langle\sigma^n\rangle=\Gal(L/L')$ and that $\mf p_s$ splits completely in $L'/K$. Now we can state Theorem B:

\begin{thm}[{cf.~\cite[Thm.~4.2]{chapdelaine2017annihilators}}]\label{thm:rootofunit}
There is a unique $\alpha\in L$ such that \linebreak $N_{L/L'}(\alpha)=1$ and such that $\eta=\alpha^y$ holds, where $y=\prod_{i=2}^{s-1}(1-\sigma^{n_i})$. This $\alpha$ is an elliptic unit of $F_I$, so that $\alpha\in \mc C_{F_I}\cap L$. Moreover, there is $\gamma\in L^\times$ such that $\alpha=\gamma^{1-\sigma^n}$.
\end{thm}

\begin{proof}
We have proven all ingredients which are used in the proof of \cite[Thm.~4.2]{chapdelaine2017annihilators}, hence we obtain the same result for function fields. 
\end{proof}

\section{Enlarging the group $\mc C_L$ of elliptic units of $L$}

We label the subfields of $L$ containing $K$ by
\begin{align*}
K=L_0\ssnq L_1\ssnq L_2\ssnq \cdots \ssnq L_k=L\ ,
\end{align*}
hence we obtain $[L_i:K]=p^i$. Moreover, we define
\begin{align*}
M_i:=\{j\in I\mid t_j>p^{k-i}\}\ .
\end{align*}
Since we have already seen that $n_s=1$, we obtain from the definition of $M_i$ that
\begin{align*}
1\in M_1\ssq M_2\ssq\cdots \ssq M_k=I\ .
\end{align*}
For $j\in M_i$ we get $p^i>\frac{p^k}{t_j}$ and with Corollary \ref{cor:1}(i) we obtain that $\mf p_j$ ramifies in $L_i$. On the other hand, if $\mf p_j$ ramifies in $L_i$, this implies that $t_j>[L:L_i]=p^{k-i}$. This shows that the conductor of $L_i$ is equal to $\mf m_{M_i}$ and so $L_i\ssq F_{M_i}$ with Proposition~\ref{prop:Gstructure} applied to $L_i$. Define 
\begin{align*}
\eta_i:=N_{F_{M_i}/L_i}(\eta_{M_i})
\end{align*}
for $i=1,...,k$, then $\eta_k=\eta\in L$ is the top generator of $\mc C_L$. \\

Now we fix $j\in \{1,...,s\}$ and let $L_i=L^{T_j}$, hence the index $i$ is determined by $t_j=p^{k-i}$. By Lemma \ref{lem:4} we get that 
\begin{align*}
\langle \sigma^{n_j}\rangle /\langle \sigma^{p^k/t_j}\rangle=\langle \sigma_j\vert_L, \sigma^{p^k/t_j}\rangle /\langle \sigma^{p^k/t_j}\rangle\ .
\end{align*}
This quotient group can be interpreted as the restriction to $L_i$, since $\sigma^{p^k/t_j}=\sigma^{p^i}$ generates $\Gal(L/L_i)$. Hence we can find a smallest positive integer $c_j$ such that $\sigma^{-c_jn_j}\vert_{L_i}=\sigma_j\vert_{L_i}$. Moreover, we see that $\mf p_j$ splits completely in $L_i/K$ if and only if $n_j=\frac{p^k}{t_j}$, in this case we get in particular that $c_j=1$ since $\sigma^{n_j}$ is already an element of the inertia group of $\mf p_j$ of $L/K$. If $\mf p_j$ does not split completely in $L_i/K$, we find that $n_j<\frac{p^k}{t_j}$ and hence $\langle\sigma^{n_j}\vert_{L_i}\rangle=\langle \sigma_j\vert_{L_i}\rangle$. In each case, we find that $p\nmid c_j$ and hence $1-\sigma^{c_jn_j}$ and $1-\sigma^{n_j}$ are associated in $\bb Z[\Gamma]$.\\

Now let $i\in \{1,...,k\}$ such that $\abs{M_i}>1$. We want to apply Theorem \ref{thm:rootofunit} to the extension $L_i/K$ and obtain an elliptic unit $\alpha_i\in \mc C_{F_{M_i}}\cap L_i$ and a number $\gamma_i\in L_i^\times$ such that
\begin{enumerate}[label=(\roman*)]
\item $\eta_i=\alpha_i^{y_i}$, where $y_i=\prod_{\substack{j\in M_i\\ 1<j<\max M_i}} (1-\sigma^{c_jn_j})$,

\item $\alpha_i=\gamma_i^{z_i}$, where $z_i=1-\sigma^{c_{\max M_i}n_{\max M_i}}$.

\end{enumerate}
Note that the new $c_j$ factors can be obtained since $1-\sigma^{n_j}$ and $1-\sigma^{c_jn_j}$ are associated. In particular we find for $\abs{M_i}=2$ that $y_i=1$ and $\alpha_i=\eta_i$ since the product is empty. For $i\in \{1,...,k\}$ with $\abs{M_i}=1$ we set $\gamma_i=\eta_i$ and $\alpha_i=\eta_i^{1-\sigma}$.

\begin{defn}
The $\bb Z[\Gamma]$-submodule $\overline{\mc C_L}$ of $\mc O_L^\times$ generated by $\mu(K)$ and $\alpha_1,...,\alpha_k$ is called the \emph{extended group of elliptic units}.
\end{defn}

\begin{thm}[{cf.~\cite[Thm.~5.2]{chapdelaine2017annihilators}}]\label{thm:enlarged}
The group of elliptic units $\mc C_L$ of $L$ is a subgroup of $\overline{\mc C_L}$ of index $[\overline{\mc C_L}:\mc C_L]=p^\nu$, where
\begin{align*}
\nu=\sum_{j=1}^k \sum_{\substack{i\in M_j\\1<i<\max M_j}} n_i\ .
\end{align*}
Moreover, setting $\varphi_L:=(\prod_{i=1}^s t_i^{n_i})\cdot \prod_{j=1}^k p^{-n_{\max M_j}}$, we get 
\begin{align*}
p^\nu=\varphi_L\cdot [L:\widetilde L]^{-1}\ 
\end{align*}
and
\begin{align*}
[\mc O_L^\times:\overline{\mc C_L}]=w_\infty^{p^k-1}\cdot \frac{h_L}{h}\cdot \varphi_L^{-1}\ .
\end{align*}

\end{thm}

\begin{proof}
We use the same proof as in \cite[Thm.~3.1]{greither2015annihilators}. Note that we need the factors $c_j$ appearing in the definition of the $\alpha_i$ here.
\end{proof}

\begin{bem}
If $p\nmid w_\infty$, we obtain $\varphi_L\mid h_L$. As in \cite[Remark 5.3]{chapdelaine2017annihilators}, this divisibility statement is really stronger than $[F_I:L]\mid h_L$ (which we obtain since $F_I/L$ is unramified). Indeed, by \cite[Prop.~3.4]{greither2015annihilators} $[F_I:L]\mid \varphi_L$ and we obtain equality if and only if $n_1=\cdots n_{s-1}=1$.
\end{bem}

\section{Semispecial numbers}

We use the same notation as before and fix $m$ which is a power of $p$ such that $p^{ks}\mid m$. We know that for a prime ideal $\mf q$ of $K$ we have
\begin{align*}
\Gal(H_\mf q/H)\cong (\mc O_K/\mf q)^\times/ \im(\mu(K))
\end{align*}
via Artin's reciprocity map. In particular, $\Gal(H_\mf q/H)$ is cyclic. This enables us to state the next

\begin{defn}
For a prime ideal $\mf q$ of $K$ such that $\abs{\mc O_K/\mf q}\equiv 1 \mod m$, we define $K[\mf q]$ to be the (unique) subfield of $H_\mf q$ containing $K$ such that $[K[\mf q]:K]=m$. For a finite field extension $M/K$, we define $M[\mf q]:=M K[\mf q]$.
\end{defn}

Note that since $\abs{\mc O_K/\mf q}\equiv 1\mod m$ and $p\nmid \abs{\mu(K)}$ we get that the order of $\Gal(H_\mf q/H)$ is divisible by $m$. Hence we get the existence and uniqueness of $K[\mf q]$ from the fact that $p\nmid h$ and \cite[Lemma 2.1]{chapdelaine2017annihilators}. Since $K[\mf q]$ is contained in $H_\mf q$ it is only ramified at $\mf q$. Moreover, since $p\nmid h$ we get that $H\cap K[\mf q]=K$ and hence it is totally ramified at $\mf q$. Finally, since $p\nmid \abs{\mc O_K/\mf q}$ we find that this ramification is tame.

\begin{defn}
Let $\mc Q_m$ be the set of all prime ideals $\mf q$ of $K$ such that
\begin{enumerate}[label=(\roman*)]

\item $\abs{\mc O_K/\mf q}\equiv 1+m\mod m^2$,

\item $\mf q$ splits completely in $L$,

\item for each $j=1,...,s$, the class of $x_j$ is an $m$-th power in $(\mc O_K/\mf q)^\times$.

\end{enumerate}
\end{defn}

%
%

Now we want to study condition (iii) in some more detail. Let $\mf q$ be such that $\abs{\mc O_K/\mf q}\equiv 1\mod m$. Since $H\cap K[\mf q]=K$, we get $\Gal(H[\mf q]/H)\cong \Gal(K[\mf q]/K)$ by restriction. The first group is the unique quotient of the cyclic group $\Gal(H_\mf q/H)$ of order $m$, hence it is obtained by factoring out $m$-th powers. Therefore we get with the Artin reciprocity map and $p\nmid w_K$
\begin{align*}
(\mc O_K/\mf q)^\times/m\cong \Gal(H[\mf q]/H)\cong \Gal(K[\mf q]/K)\ ,
\end{align*}
where the composition map takes the class of $\alpha\in \mc O_K\setminus \mf q$ to $(\alpha\mc O_K, K[\mf q]/K)$. Now the facts that $x_j\mc O_K=\mf p_j^h$ and $p\nmid h$ imply that condition (iii) is equivalent to 
\begin{align*}
(\mf p_j, K[\mf q]/K)=1\qquad \forall j=1,...,s\ .
\end{align*}

\begin{defn}\label{defn:semispecial}
A number $\varepsilon\in L^\times$ is called \emph{$m$-semispecial} if for all but finitely many $\mf q\in \mc Q_m$, there exists a unit $\varepsilon_\mf q\in \mc O_{L[\mf q]}^\times$ satisfying

\begin{enumerate}[label=(\roman*)]
\item $N_{L[\mf q]/L}(\varepsilon_\mf q)=1$,

\item if $\mf q_{L[\mf q]}$ is the product of all primes of $L[\mf q]$ above $\mf q$, then $\varepsilon$ and $\varepsilon_\mf q$ have the same image in $(\mc O_{L[\mf q]}/\mf q_{L[\mf q]})^\times/(m/p^{k(s-1)})$.

\end{enumerate}
\end{defn}

Since each $\mf q\in \mc Q_m$ is totally ramified in $K[\mf q]/K$ and splits completely in $L/K$, we find that $L[\mf q]/L$ is totally ramified at each prime above $\mf q$ and we obtain \linebreak $L\cap K[\mf q]=K$. Therefore, the two restriction maps $\Gal(L[\mf q]/L)\lra \Gal(K[\mf q]/K)$ and $\Gal(L[\mf q]/K[\mf q])\lra \Gal(L/K)$ are isomorphisms.

\begin{thm}[{cf.~\cite[Thm.~6.4]{chapdelaine2017annihilators}}]\label{thm:19}
The elliptic unit $\alpha\in \mc C_{F_I}\cap L$ from Theorem~\ref{thm:rootofunit} is $m$-semispecial.
\end{thm}

\begin{proof}
Recall that $\alpha$ is a $y$-th root of the top generator $\eta$ of $\mc C_L$. We need to show that for almost all $\mf q\in \mc Q_m$, there exists an $\varepsilon_\mf q$ satisfying the conditions (i) and (ii) of Definition \ref{defn:semispecial}. In fact, we can construct such an $\varepsilon_\mf q$ for \emph{each} $\mf q\in \mc Q_m$ with the methods of \cite[Thm.~6.4]{chapdelaine2017annihilators}. There is only a slight change in the result which implies the congruence relation. Therefore we skip the proof of the theorem here and refer to \cite{chapdelaine2017annihilators} once again. Our version of Proposition~6.6 is stated below.
\end{proof}

\begin{prop}[{cf.~\cite[Prop.~6.6]{chapdelaine2017annihilators}}]\label{prop:21}
Let $\mf q\in \mc Q_m$, $Q:=\abs{\mc O_K/\mf q}$ and let $\mf q_{L[\mf q]}$ be the product of all primes of $L[\mf q]$ above $\mf q$. Then 
\begin{align*}
\hat{\eta}^{Q(1-\sigma)}\equiv \eta^{(1-\sigma)\frac{Q-1}{m}}\mod \mf q_{L[\mf q]}\ ,
\end{align*}
where $\eta$ is the top generator of $\mc C_L$ and $\hat{\eta}$ is the top generator of $\mc C_{L[\mf q]}$.
\end{prop}

\begin{proof}
Let $x\in \mc O_K$, such that $x\mc O_K=\mf q^h$. Let $K_m:=K(\zeta_m)$, where $\zeta_m$ is a primitive $m$-th root of unity. Then $K_m/K$ is a constant field extension and hence, it is unramified everywhere. Moreover, it is an abelian extension. Now we can define $M:=K_m(x^{1/p})$, and since $\mc O_K^\times=\mu(K)$, $p\nmid \abs{\mu(K)}$ and $K_m$ contains a primitive $p$-th root of unity, this definition is independent of the choice of the generator $x$ and of its $p$-th root. Then $M/K$ is a Galois extension. We claim that $x$ is not a $p$-th power in $K_m$. If $x=\alpha^p$, then the valuation of $x$ at $\mf q$ would be $p$-times the valuation of $\alpha$ at $\mf q$ since $K_m/K$ is unramified. But $x\mc O_K=\mf q^h$ and since $p\nmid h$, this is a contradiction.
Hence the extension $M/K_m$ is cyclic of degree $p$. For finishing the proof, we need the next

\begin{lem}\label{lem:22}
Let $\mf q\in \mc Q_m$ and let $\sigma$ be the unique generator of $\Gal(L[\mf q]/K[\mf q])$ which restricts to the original generator of $\Gal(L/K)$. Then there exists a prime $\mf l$ of $K$ such that
\begin{enumerate}[label=(\roman*)]
\item $\abs{\mc O_K/\mf l}\equiv 1\mod m$,

\item $\mf l$ is unramified in $L[\mf q]$ and $(\mf l, L[\mf q]/K)=\sigma^{-1}$,

\item $\mf q$ is inert in $K[\mf l]/K$.

\end{enumerate}
\end{lem}
\begin{proof}
By an explicit analysis of the Galois automorphisms, one can check that $K_m/K$ is an abelian extension whereas $M/K_m$ is not. Since $[M:K_m]=p$, there are no intermediate fields and hence $K_m/K$ is the maximal abelian subextension of $M$. This implies that $M\cap L[\mf q]=K_m\cap L[\mf q]$, since $L[\mf q]/K$ is an abelian extension. Since $K_m\cap L[\mf q]$ is unramified and $p\nmid h$, we find $K_m\cap L[\mf q]=K$. Then there exists a $\tau\in \Gal(L[\mf q]\cdot M/K)$ which restricts to $\sigma^{-1}\in \Gal(L[\mf q]/K)$ and to a generator of $\Gal(M/K_m)\ssq \Gal(M/K)$.\\

Using a variant of \v{C}ebotarev's Density Theorem (cf.~\cite[Thm.~9.13B]{rosen2013number}), we see that there exists a prime $\mf l$ such that the Frobenius of $\mf l$ is the conjugacy class of $\tau$ and $\abs{\mc O_K/\mf l}\equiv 1\mod m$. Then the first two conditions are satisfied and it remains to prove that $\mf q$ is inert in $K[\mf l]$.\\

Since $\tau$ acts as the identity on $K_m$, we find that $\mf l$ splits completely in $K_m/K$. Let $\mf L$ be a prime of $K_m$ over $\mf l$, then $\mc O_{K_m}/\mf L\cong \mc O_K/\mf l$. Moreover, since $$\langle \tau\vert_M\rangle=\Gal(M/K_m)	\cong \bb Z/p\bb Z,$$ $\mf L$ must be inert in $M$. It is easily seen that $\mc O_M/\mf L\mc O_M\cong (\mc O_{K_m}/\mf L)[\xi]$, where $\xi$ is the class of $x^{1/p}$ modulo $\mf L\mc O_M$. If $x$ was a $p$-th power in $(\mc O_{K_m}/\mf L)^\times$, this extension would be trivial, hence the inertia degree of $\mf L$ would be one. This is a contradiction, since $\mf L$ is inert in $M$, so we have shown that $x$ cannot be a $p$-th power in $(\mc O_K/\mf l)^\times$.\\

Recall that we get $(\mc O_K/\mf l)^\times/m\cong \Gal(K[\mf l]/K)$ from Artin's Reciprocity Theorem and $p\nmid w_K$. Since $x$ is not a $p$-th power in $(\mc O_K/\mf l)^\times$, it follows that the Frobenius $(x\mc O_K, K[\mf l]/K)=(\mf q, K[\mf l]/K)^h$ is not a $p$-th power in $\Gal(K[\mf l]/K)$. But since $\Gal(K[\mf l]/K)$ is cyclic of order $m$ and $p\nmid h$, we obtain that $(\mf q, K[\mf l]/K)$ generates $\Gal(K[\mf l]/K)$ and hence $\mf q$ is inert in $K[\mf l]$.
\end{proof}

Using the prime $\mf l$ satisfying the conditions of the previous lemma, we can define the elliptic units
\begin{align*}
\eta_\mf l&:=N_{H_{\mf l\mf m_I}/L[\mf l]}(\alpha_{\mf l\mf m_I})\ ,\\
\hat{\eta}_{\mf l}&:=N_{H_{\mf l\mf q\mf m_{I}}/L[\mf q\mf l]}(\alpha_{\mf l\mf q\mf m_{I}})\ ,
\end{align*}
where $L[\mf q\mf l]$ is the compositum of $L[\mf q]$ and $L[\mf l]$. Using the norm relation, we find
\begin{align*}
N_{L[\mf q\mf l]/L[\mf l]}(\hat{\eta}_\mf l)&=\eta_\mf l^{1-\sigma_\mf q^{-1}}\ ,\\
N_{L[\mf q\mf l]/L[\mf q]}(\hat{\eta}_\mf l)&=\hat{\eta}^{1-\sigma_\mf l^{-1}}=\hat{\eta}^{1-\sigma}\ ,\\
N_{L[\mf l]/L}(\eta_\mf l)&=\eta^{1-\sigma_\mf l^{-1}}=\eta^{1-\sigma}\ ,
\end{align*}
where $\sigma_\mf q=(\mf q, L[\mf l]/K)$ and $\sigma_\mf l=(\mf l,L[\mf q]/K)=\sigma^{-1}$ by condition (ii).\\

Since $\mf q\in \mc Q_m$, $\mf q$ splits completely in $L/K$ and by condition (iii), the primes of $L$ above $\mf q$ are inert in $L[\mf l]/L$. Then each prime of $L[\mf q]$ above $\mf q$ must also be inert in $L[\mf q\mf l]/L[\mf q]$. Moreover, since each prime above $\mf q$ is unramified in $L[\mf l]/L$ and totally ramified in $L[\mf q]/L$, it is also totally ramified in $L[\mf q\mf l]/L[\mf l]$, hence the product of all primes of $L[\mf q\mf l]$ above $\mf q$ is given by $\mf q_{L[\mf q]}\mc O_{L[\mf q\mf l]}$. Therefore, we get the following isomorphism of rings
\begin{align*}
\mc O_{L[\mf q\mf l]}/\mf q_{L[\mf q]}\mc O_{L[\mf q\mf l]}\cong \mc O_{L[\mf l]}/\mf q\mc O_{L[\mf l]}\ .
\end{align*}
Since $L[\mf q]$ and $L[\mf l]$ are linearly disjoint over $L$ and $\mf q$ splits completely in $L/K$, we can extend $\sigma_\mf q\in \Gal(L[\mf l]/K)$ to $L[\mf q\mf l]$ such that this extension (also denoted by $\sigma_\mf q$) restricts to the identity on $L[\mf q]$. In particular, $\sigma_\mf q$ generates $\Gal(L[\mf q\mf l]/L[\mf q])$.\\

From the above isomorphism, we get that $\sigma_\mf q$ acts as raising to the $Q$-th power on $\mc O_{L[\mf q\mf l]}/\mf q_{L[\mf q]}\mc O_{L[\mf q\mf l]}$. Moreover, the group $\Gal(L[\mf q\mf l]/L[\mf l])$ is the inertia group at $\mf q$ and acts trivially on $\mc O_{L[\mf q\mf l]}/\widetilde{\mf q}\mc O_{L[\mf q\mf l]}$. Therefore, we can express the action of the norms $N_{L[\mf q\mf l]/L[\mf l]}$ and $N_{L[\mf q\mf l]/L[\mf q]}$ on $\mc O_{L[\mf q\mf l]}/\mf q_{L[\mf q]}\mc O_{L[\mf q\mf l]}$ as raising to the power $m$ respectively to the power $\sum_{i=0}^{m-1} Q^i$. Since $Q\equiv 1\mod m$, there exists a positive integer $r$ such that $\sum_{i=0}^{m-1} Q^i=mr$. Combining our results, we get that
\begin{align*}
\hat{\eta}^{Q(1-\sigma)}&\equiv \hat{\eta}_{\mf l}^{Qmr}\equiv \eta_\mf l^{Qr(1-\sigma_\mf q^{-1})}\equiv \eta_\mf l^{r(Q-1)}\equiv (\eta_\mf l^{mr})^{\frac{Q-1}{m}}\\
&\equiv \eta^{(1-\sigma)\frac{Q-1}{m}}\mod \mf q_{L[\mf q]}\mc O_{L[\mf q\mf l]}\ .
\end{align*}
Since the natural map $\mc O_{L[\mf q]}/\mf q_{L[\mf q]}\lra \mc O_{L[\mf q\mf l]}/\mf q_{L[\mf q]}\mc O_{L[\mf q\mf l]}$ is injective, we obtain the desired result.

\end{proof}

\section{Annihilating the ideal class group}\label{sec:CKend}

Using the same notation as before, we define
\begin{align*}
\mu_i:=n_{\max M_i}\ .
\end{align*}
This is always a power of $p$ and since $M_i\ssq M_{i+1}$, we get $\mu_i\leq \mu_{i+1}$. We call an index $i\in\{1,...,k-1\}$ a \emph{jump} if $\mu_i<\mu_{i+1}$. Further, we declare $0$ and $k$ to be jumps and set $\mu_0=0$. Then we get the next

\begin{lem}[{cf.~\cite[Lemma 7.1]{chapdelaine2017annihilators}}]\label{lem:Zbasis}
Let $0=s_0<s_1<...<s_\kappa=k$ be the ordered sequence of all jumps. Then the set
\begin{align*}
\bigcup_{t=1}^\kappa \{\alpha_{s_t}^{\sigma^i}\mid 0\leq i<p^{s_t}-p^{s_{t-1}}\}
\end{align*}
is a $\bb Z$-basis of $\overline{\mc C_L}/\mu(K)$.
\end{lem}

\begin{proof}
See \cite[Lemma 5.1]{greither2015annihilators}.
\end{proof}

With this basis, we obtain our next result:

\begin{lem}[{cf.~\cite[Lemma 7.2]{chapdelaine2017annihilators}}]
Let $r$ be the highest jump less than $k$, i.e. $\mu_r<\mu_{r+1}=n_s$. Assume that $\rho\in \bb Z[\Gamma]$ is such that $\alpha_k^\rho\in \overline{\mc C_{L_r}}$. Then
\begin{align*}
(1-\sigma^{p^r})\rho=0\ .
\end{align*}
\end{lem}

\begin{proof}
See \cite[Lemma 7.2]{chapdelaine2017annihilators}.
\end{proof}

Now we need an additional condition on the $p$-power $m$. We already know that $(m,q)=1$, since $p\nmid q$, so $q\in (\bb Z/m\bb Z)^\times$. Let $d$ denote the order of $q$ in $(\bb Z/m\bb Z)^\times$, then there exists $i\geq 0$ and $b\in \bb Z$ with $p\nmid b$ such that
\begin{align*}
q^d-1=b\cdot p^im.
\end{align*}
If we define $m':=p^im$, we still have $p^{ks}\mid m'$ and $d$ is the order of $q$ modulo $m'$, so we can assume without loss of generality that $i=0$. Now we can define $f$ to be the order of $q$ in $(\bb Z/m^2\bb Z)^\times$. Then an easy computation shows:

\begin{lem}\label{lem:choiceofm}
We have $m\mid \frac{f}{d}$.
\end{lem}

\begin{thm}\label{thm:A}
Let $m$ be a power of $p$ such that $m\mid \frac{f}{d}$ and $V\ssq L^\times/m$ a finitely generated $\bb Z_p[\Gamma]$-submodule. Without loss of generality we can choose representatives of generators of $V$ which belong to $\mc O_L$. Suppose there is a map $z:V\lra (\bb Z/m\bb Z)[\Gamma]$ of $\bb Z_p[\Gamma]$-modules such that $z(V\cap K^\times)=0$, where $V\cap K^\times$ means $V\cap (K^\times (L^\times)^m/(L^\times)^m)$. Then for any $\mf c\in Cl(\mc O_L)_p$, there exist infinitely many primes $\mf Q$ in $L$ such that:
\begin{enumerate}[label=(\roman*)]
\item $\mf q:=\mf Q\cap K$ is completely split in $L/K$,

\item $[\mf Q]=\mf c$, where $[\mf Q]$ is the projection of the ideal class of $\mf Q$ into $Cl(L)_p$,

\item $Q:=\abs{\mc O_L/\mf Q}\equiv 1+m\mod m^2$,

\item for each $j=1,...,s$, the class of $x_j$ is an $m$-th power in $(\mc O_K/\mf q)^\times$,

\item no prime above $\mf q$ is contained in the support of the generators of $V$ and there is a $\bb Z_p[\Gamma]$-linear map $\varphi:(\mc O_L/\mf q\mc O_L)^\times/m\lra (\bb Z/m\bb Z)[\Gamma]$ such that the diagram
\[\begin{tikzcd}
V\arrow{r}{z}\arrow{d}{\psi} & (\bb Z/m\bb Z)[\Gamma]\\
(\mc O_L/\mf q\mc O_L)^\times/m\arrow{ru}{\varphi} & 
\end{tikzcd}\]
commutes, where $\psi$ corresponds to the reduction map.

\end{enumerate}
\end{thm}

\begin{bem}
The reduction map $\psi$ is defined on the chosen set of generators: Let $x\in \mc O_L$ be a representative of such a generator, then $\overline{x}$ is the class of $x\in \mc O_L/\mf q\mc O_L$. Since no prime above $\mf q$ is contained in the support of $x$, we get $\overline{x}\in (\mc O_L/\mf q\mc O_L)^\times$. Hence we can set $\psi(x)$ to be the class of $\overline{x}$ in $(\mc O_L/\mf q\mc O_L)^\times/m$. This yields a well-defined $\bb Z_p[\Gamma]$-homomorphism.
\end{bem}

\begin{proof}
The proof is essentially the same as in \cite[Thm.~17]{greither2004annihilators}. We only point out some changes which are necessary in the function field case:
\begin{itemize}
\item In the proof of \cite[Lemma 18(i) and (ii)]{greither2004annihilators}, we only get an isomorphism to a subgroup $H$ of $(\bb Z/m^2\bb Z)^\times$ of order $f$. Nevertheless, we can choose $q$ as a generator of $H$ and hence prove the vanishing of the coinvariants since $p\nmid q-1$.

\item For part (iii) of \cite[Lemma 18]{greither2004annihilators} we consider the splitting field of $\infty$ in $L\bb F_{q^f}$ which is the unique subextension of degree $(f,d_\infty)$ (see \cite[Prop.~8.13]{rosen2013number}). Then the claim follows since $d_\infty\mid h$ and $p\nmid h$, so $p\nmid d_\infty$.

\item In the last step of the construction of the element $\tau$ which is used for \v{C}ebotarev's Density Theorem (cf.~\cite[Thm.~9.13A]{rosen2013number}), we need an element of order $m$ in $\Gal(L\bb F_{q^f}/L\bb F_{q^d})$. At this point we need $m\mid \frac{f}{d}$ from Lemma \ref{lem:choiceofm}.

\item Condition (iii) uses that $\zeta_{m^2}$ is an element of the constant field of $L\bb F_{q^f}$.
\end{itemize}
\end{proof}

For the desired annihilation result, we need the next

\begin{thm}[{cf.~\cite[Thm.~(5.1)]{rubin1987global}}]\label{thm:B}
Let $\mf q$ be a prime of $K$ which splits completely in $L$, set $Q:=\abs{\mc O_K/\mf q}$. Let $M$ be a finite extension of $L$ which is abelian over $K$ and such that in $M/L$, all primes above $\mf q$ are totally tamely ramified and no other primes ramify. Write $\mf q_M$ for the product of all primes of $M$ above $\mf q$ and let $\mc A$ denote the annihilator in $(\bb Z/(Q-1)\bb Z)[\Gamma]$ of the cokernel of the reduction map
\begin{align*}
\{\varepsilon\in \mc O_M^\times\mid N_{M/L}(\varepsilon)=1\}\lra (\mc O_M/\mf q_M)^\times\ .
\end{align*}
Write $w:=\frac{Q-1}{[M:L]}$. Then $\mc A\ssq w(\bb Z/(Q-1)\bb Z)[\Gamma]$ and for every prime $\mf Q$ of $L$ above $\mf q$, $w^{-1}\mc A$ annihilates the ideal class of $\mf Q$ in $Cl(\mc O_L)/[M:L]$.
\end{thm}

\begin{proof}
The proof of \cite[Thm.~(5.1)]{rubin1987global} also works for function fields.
\end{proof}

The above theorems are the main ingredients for proving

\begin{thm}\label{thm:C}
Let $m$ be a power of $p$ divisible by $p^{ks}$ such that $m\mid \frac{f}{d}$. Assume that $\varepsilon\in \mc O_L$ is $m$-semispecial and let $V\ssq L^\times/m$ be a finitely generated $\bb Z[\Gamma]$-module. Suppose that the class of $\varepsilon$ belongs to $V$. Let $z:V\lra (\bb Z/m\bb Z)[\Gamma]$ be a $\bb Z[\Gamma]$-linear map such that $z(V\cap K^\times)=0$. Then $z(\varepsilon)$ annihilates $Cl(\mc O_L)_p/(m/p^{k(s-1)})$.
\end{thm}

\begin{proof}
See \cite[Thm.~12]{greither2004annihilators}.
\end{proof}

The main result of this article is Theorem C:

\begin{thm}\label{thm:main}
Let $r$ be the highest jump less than $k$. Then we have
\begin{align*}
\Ann_{\bb Z[\Gamma]}((\mc O_L^\times/\overline{\mc C_L})_p)\ssq \Ann_{\bb Z[\Gamma]}((1-\sigma^{p^r})Cl(\mc O_L)_p)\ .
\end{align*}
The number $r$ is determined by $p^{k-r}=\max\{t_j\mid j\in J\}$, where $J=\{j\in \{1,...,s\}\mid n_j=n_s\}$.
\end{thm}

\begin{proof}
The proof of \cite[Thm.~7.5]{chapdelaine2017annihilators} can be used without any changes, since we have proved all the main ingredients in the Theorems \ref{thm:A}, \ref{thm:B} and \ref{thm:C}.
\end{proof}

\newpage

\bibliographystyle{abbrv}

\footnotesize

Pascal Stucky, \textsc{Department of Mathematics, Ludwig-Maximilians-Universit\"at M\"unchen}\par\nopagebreak
  \textit{E-mail address: }\texttt{stucky@math.lmu.de}

\end{document}